\documentclass[11pt,reqno]{amsart}
\usepackage{microtype}
\usepackage[margin=3cm]{geometry}

\usepackage[shortlabels]{enumitem}
\setlist[enumerate]{label={(\arabic*)}}

\usepackage{amssymb}
\usepackage[dvipsnames]{xcolor}
\usepackage
[colorlinks=true,linkcolor=Maroon,citecolor=OliveGreen]
{hyperref}
\usepackage[abbrev, alphabetic]{amsrefs}  
\usepackage[capitalize]{cleveref} 
\crefname{equation}{}{}

\usepackage[norefs, nocites]{refcheck}

\makeatletter
\newcommand{\refcheckize}[1]{%
  \expandafter\let\csname @@\string#1\endcsname#1%
  \expandafter\DeclareRobustCommand\csname relax\string#1\endcsname[1]{%
    \csname @@\string#1\endcsname{##1}\wrtusdrf{##1}}%
  \expandafter\let\expandafter#1\csname relax\string#1\endcsname
}
\makeatother

\refcheckize{\cref}
\refcheckize{\Cref}

\numberwithin{equation}{section}

\newtheorem{lemma}{Lemma}[section]
\newtheorem{theorem}[lemma]{Theorem}
\newtheorem{proposition}[lemma]{Proposition}
\newtheorem{corollary}[lemma]{Corollary}

\theoremstyle{definition}
\newtheorem{remark}[lemma]{Remark}
\newtheorem{notation}[lemma]{Notation}
\newtheorem*{choice}{Choice of $A$}

\newcommand\opr[1]{\operatorname{#1}}

\newcommand{\eps}{\epsilon}

\def\Zint{\mathbf{Z}}
\def\F{\mathbf{F}}

\def\Z{\mathrm{Z}}
\def\fpr{\mathrm{fpr}}
\def\fp{\mathrm{fp}}

\newcommand\floor[1]{\left\lfloor{#1}\right\rfloor}

\newcommand{\gen}[1]{\langle{#1}\rangle}

\newcommand{\GL}{\opr{GL}}
\newcommand{\SL}{\opr{SL}}

\newcommand{\SO}{\opr{SO}}
\newcommand{\SU}{\opr{SU}}
\newcommand{\GU}{\opr{GU}}
\newcommand{\Or}{\opr{O}}
\newcommand{\PGL}{\opr{PGL}}

\newcommand{\Sp}{\opr{Sp}}
\newcommand{\POm}{\opr{P\Omega}}
\newcommand{\PO}{\opr{PO}}
\newcommand{\PSL}{\opr{PSL}}
\newcommand{\PSU}{\opr{PSU}}
\newcommand{\PSp}{\opr{PSp}}
\newcommand{\PSO}{\opr{PSO}}
\newcommand{\PGamL}{\opr{P\Gamma L}}
\newcommand{\GamL}{\opr{\Gamma L}}

\newcommand{\Aut}{\opr{Aut}}

\def\sm{\smallsetminus}

\def\dim{\opr{dim}}

\usepackage{todonotes}

\begin{document}

\title[Probabilistic Generation]{Probabilistic Generation of Finite Almost Simple Groups}

\author[J. Fulman]{Jason Fulman}
\address{Department of Mathematics, University of Southern California, Los Angeles, CA 90089-2532, USA}
\email{fulman@usc.edu}

\author[D. Garzoni]{Daniele Garzoni} 
\address{Department of Mathematics, University of Southern California, Los Angeles, CA 90089-2532, USA}
\email{garzoni@usc.edu}

\author[R. M. Guralnick]{Robert M. Guralnick}
\address{Department of Mathematics, University of Southern California, Los Angeles, CA 90089-2532, USA}
\email{guralnic@usc.edu}

\thanks{Fulman was partially supported by Simons Foundation Grants 400528 and 917224. Guralnick was supported
by the NSF grant DMS-1901595 and a Simons Foundation Fellowship 609771. }



\ \date{March 25, 2024}

\begin{abstract}
We prove that if $G$ is a sufficiently large finite almost simple group of Lie type, then given a fixed nontrivial element $x \in G$
and a coset of $G$ modulo its socle, the probability that $x$ and a random element of the coset generate a subgroup containing
the socle is uniformly bounded away from $0$ (and goes to $1$ if the field size goes to $\infty$). This is new even if $G$ is simple.   Together with results of Lucchini and Burness--Guralnick--Harper, this proves a conjecture of Lucchini and has an application to profinite groups. A key step in the proof is the determination of the limits for the proportion of elements in a classical group which fix no subspace of any bounded dimension.
\end{abstract}

\maketitle

\section{Introduction}

The problem of determining the smallest cardinality of a generating set for finite groups has a very long history, and has had various applications -- e.g.,  Thompson proved that a finite group is solvable if and only if every two generated
subgroup is solvable.   The critical case 
for handling these problems is when the group is simple or almost simple. (Recall that a finite group $G$ is \textit{almost simple} if there exists a nonabelian finite simple group $S$ with $S\le G \le \Aut(S)$.) Combining the classification of simple groups together
with Steinberg's result \cite{St} shows that every finite simple group can be generated by two elements.  

In recent  years, probabilistic methods have been introduced.  In particular, it was shown in  \cite{Di,KL,LSh} that the probability that two random elements of a finite simple group generate the group goes to $1$ as the  order of the group goes to infinity.

A finite group is said to be $\frac{3}{2}$-generated if every nontrivial element belongs to a generating pair. Already in \cite{St} it was asked whether every finite simple group is $\frac{3}{2}$-generated.  This (and more) was 
proved in \cite{guralnick2000kantor, breuer2008probabilistic}. Recently Burness, Guralnick and Harper \cite{BGH} showed that a finite group is $\frac{3}{2}$-generated if and only if every proper quotient is cyclic.

 Kantor and Lubotzky \cite{KL} asked whether there was a probabilistic version of $\frac{3}{2}$-generation. This is not the case for alternating groups: If $x\in A_n$ moves only a bounded number of points, it is clear that the probability that $x$ and a random element of $A_n$ generate
 a transitive group goes to $0$ as $n \rightarrow \infty$.
 
 In this paper, we prove that, on the contrary, a probabilistic version holds for simple groups of Lie type. More generally, we consider almost simple groups of Lie type, by fixing any nontrivial element and picking random elements from a fixed coset of the socle. Together with \cite{BGH2,Lu} this solves a conjecture of Lucchini and has an application to profinite groups; see \Cref{cor:lucchini,cor:profinite}.

\begin{theorem}
\label{t:main}
    There exists an absolute constant $\eps>0$ such that the following holds. Let $S$ be a finite simple group of Lie type of large enough order, and let $x,y\in \Aut(S)$ with $x\neq 1$. Then the probability that $x$ and a random element of $Sy$ generate $\gen{S,x,y}$ is at least $\eps$.
\end{theorem}

If $S$ is defined over $\F_q$ with $q$ large, and if $x,y\in S$, then \Cref{t:main} was proved in Liebeck--Shalev \cite{liebeck1999shalev} and Guralnick--Kantor--Saxl \cite{guralnick1994probability} (and the probability tends to $1$ as $q\to \infty$). For $q$ large, we extend this result to the case of arbitrary $x$ and $y$. We stress, however, that the key contribution of this paper is the case of classical groups over small fields (which is the most difficult case), also for $x,y\in S$; see \Cref{subsec:idea_proof} for more details. It is known that, for bounded $q$, the probability in \Cref{t:main} can be bounded away from $1$, so that \Cref{t:main} cannot be improved.

The fact that  $\gen{S,x,y}$ is $2$-generated was proved by Dalla Volta--Lucchini \cite{dallavolta1995lucchini}. A variant of the proof of the $\frac{3}{2}$-generation result of \cite{BGH} gives that if $S$ is a nonabelian simple group and $x, y \in \Aut(S)$  with $x \ne 1$, then there exists $s \in S$ such that $\gen{x,sy}=\gen{S,x,y}$; see \cite{BGH2}. In particular, this  proves \Cref{t:main} for sporadic groups and for the finitely many groups of Lie type excluded from the statement.  

As already remarked, the probabilistic statement fails for alternating groups. Nevertheless, Lucchini \cite{Lu} proved that if $x,y\in S_n$ with $x\neq 1$, then the probability that $x$ and a random element of $A_ny$ generate a nonsolvable group is bounded away from zero. In particular, we deduce from \cite{Lu} (alternating groups), \Cref{t:main} (groups of Lie type of large enough order) and \cite{BGH2} (sporadic groups and remaining groups of Lie type) the following, which was conjectured by Lucchini in \cite[Conjecture 5]{Lu}.

\begin{corollary}
	\label{cor:soluble}
	\label{cor:lucchini} \label{c:solvable}  There exists an absolute constant $\eps>0$ such that the following holds. Let $S$ be a nonabelian finite simple group  and let $x,y\in \Aut(S)$ with $x\neq 1$. Then the probability that $x$ and a random element of $Sy$ generate a nonsolvable group is at least $\eps$. 
\end{corollary}

In \cite[Theorem 4]{Lu} it was shown  that \Cref{cor:lucchini} implies the following result on profinite groups, stating that for an element $g\in G$, $g$ and a random element have a chance to generate a prosolvable group if and only if $g$ centralizes almost all nonabelian chief factors of $G$. (Here the probability on $G$ is the Haar measure.)

\begin{corollary}
	\label{cor:profinite}
Let $G$ be a profinite group and let $g\in G$. Then, the following are equivalent:
\begin{itemize}
	\item[(i)] The probability that $g$ and a random element of $G$ generate a prosolvable group is positive;
	\item[(ii)] There exists $C\ge 1$ such that $g$ centralizes all but at most $C$ nonabelian chief factors of $G/N$ for every every open normal subgroup $N$ of $G$.
\end{itemize}
\end{corollary}

One might ask whether the probability of generating a nonsolvable group goes to $1$ as the order of $S$ goes to $\infty$ (cf. \Cref{cor:soluble}). This is not needed for the application to profinite groups (\Cref{cor:profinite}), but is an interesting problem which we plan to address in future work.

A key ingredient in the proof of \Cref{t:main} is a result of independent interest.  We consider the proportion of elements in a classical group of dimension $n$ over the field of size $q$ which fix no subspace
of dimension at most $t$.   For $q$ and $t$ fixed, we prove that the limit as $n \rightarrow \infty$ exists and is strictly between $0$ and $1$.  We elaborate on this in the next subsection.

\subsection{Idea of the proof}
\label{subsec:idea_proof}

Let us explain the method of proof of \Cref{t:main}. Let $M$ be a maximal subgroup of $\gen{S,x,y}$. We will see (\Cref{l:fixed_point_ratio}) that the probability that $x$ and a random element of $Sy$ are contained in a conjugate of $M$ is at most the fixed point ratio of $x$ on the cosets of $M$. Using fixed point ratio estimates from Guralnick--Kantor \cite{guralnick2000kantor}, Liebeck--Saxl \cite{liebeck1991saxl} and Liebeck--Shalev \cite{liebeck1999shalev}, this leaves us with classical groups over fields of size $2$ and $3$ and $M$ a subspace subgroup.

We remark at once that this is not a critical advancement, in that handling small fields is the most difficult case, and solving the problem for $\SL_n(2)$, say, is as difficult as solving the problem for $\PSL_n(q)$ for every $q$. 

Indeed, the remaining cases require much more work, and we make use of a new method in order to address them. Specifically, we pick random elements from a certain subset $A$ of $Sy$, and we make use of a key lemma (\Cref{l:expectation}), asserting that we can multiply $\fpr(x,G/M)$ by the expected number of fixed points of an element of $A$ on $G/M$, namely, by
\begin{equation}
\label{eq:average_intro}
\frac{1}{|A|}\sum_{g\in A}\fp(g,G/M).
\end{equation}
(Note that the aforementioned \Cref{l:fixed_point_ratio} is just a special case, since the expected number of fixed points of an element of $Sy$ is $1$.) The success of this strategy depends on being able to prove two facts:
\begin{itemize}
    \item[(i)] $|A|/|S|$ is bounded away from zero.
    \item[(ii)] The expectation \eqref{eq:average_intro} is small.
\end{itemize}

Let $A=A(t)$ be the set of elements of $Sy$ fixing no space of dimension at most $t$ (with suitable necessary changes for some choices of $Sy$). In order to prove (i) we will use generating functions. We will prove that the limit as $n\to \infty$ with $q$ and $t$ fixed of $|A|/|S|$ exists and is between $0$ and $1$. This generalizes work of Neumann--Praeger \cite{NP}, which addressed the case $t=1$. 

We will then prove (ii). Evidently, when $M$ is a stabilizer of a space of dimension $k\le t$, the average \eqref{eq:average_intro} is simply zero. What is more, the average over $Sy$ is $1$, for any $k$. Given that $A$ is obtained from $Sy$ by removing elements fixing (loosely speaking) many subspaces, we expect that the average over $A$ be smaller than $1$ also in the case $k>t$.  We will prove this fact, by computing the asymptotic of \eqref{eq:average_intro} in the regime $q$ and $t$ fixed and $n\to \infty$. The proof depends essentially on the existence of the limits discussed in the previous paragraph, as well as on some of their properties.

We conclude with a comment. In order to make the proof explicit, we will choose an explicit value for $t$ (and $t\le 4$ in all cases). However, it is worth explaining the proof from a somewhat more satisfying perspective. By \cite{guralnick2000kantor},
\begin{equation}
	\label{eq:intro_kantor}
\sum_M \fpr(x,G/M) \ll \sum_{k=1}^\infty \frac{1}{q^k},
\end{equation}
the sum running over all conjugacy classes of subspace subgroups. By taking $t$ sufficiently large but fixed, and using \Cref{l:expectation} and (i) and (ii), we can just cancel all terms $1/q^k$ with $k$ as large as we want, and since $\sum_k 1/q^k$ is a convergent series, we can make the right-hand side of \eqref{eq:intro_kantor} as small as we want. (One has to be slightly careful because there are cases where some contributions $1/q$ or $1/q^2$ simply cannot be cancelled, due to the nature of the coset $Sy$.)

\subsection{Notation} If $G$ is a group acting on a finite set $\Omega$ and $g\in G$, we denote by $\fp(g,\Omega)$ the number of fixed points of $g$ on $\Omega$, and we denote by $\fpr(g,\Omega)$ the \textit{fixed point ratio} $\fp(g,\Omega)/|\Omega|$. If $H$ is a subgroup of $G$ we denote by $G/H$ the set of right cosets of $H$ in $G$.

We write $f\ll g$ if there exists an absolute constant $C>0$ such that $f \le Cg$. If $C$ depends on a parameter, say $r$, then we write $f\ll_r g$.

We write $(n,m)$ for the greatest common divisors of the positive integers $n$ and $m$.

\section{Generating functions}
\label{sec:generating_functions}

Neumann and Praeger \cite{NP} studied the proportion of eigenvalue free elements in a finite classical group $G$. We extend this in two
directions. First, we consider proportions of elements in cosets of $\SL_n(q)$ in $\GL_n(q)$ and find that the limiting proportions are the same
as for $\GL_n(q)$. Although we do not need it, the same is true for proportions of elements in cosets of $\SU_n(q)$ in $\GU_n(q)$; see the appendix. Second, for linear, unitary, symplectic, and orthogonal groups, we investigate proportions of elements whose characteristic polynomial has no irreducible factors of degree at most some fixed $t$ (so the case $t=1$ corresponds to eigenvalue free elements).

\subsection{Cosets of $\SL_n(q)$ in $\GL_n(q)$} \label{cosetA}

The main purpose of this section is to show that for any coset of $\SL_n(q)$ in $\GL_n(q)$, the fixed $q$, $n \rightarrow \infty$ proportion of 
elements whose characteristic polynomial has no irreducible factors of degree at most $t$ converges to the fixed $q$, $n \rightarrow \infty$ proportion of the corresponding elements of $\GL_n(q)$. From \cite{NP}, when $t=1$, the fixed $q$, $n \rightarrow \infty$ proportion of eigenvalue free elements of $\GL_n(q)$ is equal to \[ \left( \prod_{i \geq 1} (1-1/q^i) \right)^{q-1}, \] and converges to $1/e$ as $q \rightarrow \infty$. We will calculate the limiting proportion for general $t$ and will show that for $q$ a prime power, it is uniformly bounded away from $0$ and $1$.

We require some definitions. Let $\F_q$ be a finite field of size $q$. We let $\Phi_q^+$ denote the set of of monic irreducible polynomials $\phi$ in $\F_q[z]$ other than the polynomial $z$. We let $deg(\phi)$ denote the degree of $\phi$. Fix a generator $\zeta$ of the multiplicative group $\F_q^{\times}$. For $\alpha \in \F_q^{\times}$, define $r(\alpha)$ to be the element of $Z_{q-1}$ such that $\zeta^{r(\alpha)}=\alpha$. For a polynomial $g$, define $r(g)=r((-1)^{deg(g)} g(0))$. We let $\Lambda$ denote the set of partitions of all non-negative integers $n$.

The following three known lemmas will be useful.

\begin{lemma} \label{fir} Suppose that $g(u) = \sum_{n \geq 0} a_n u^n$ where $g(u) = f(u)/(1-u)$ for $|u|<1$. If $f(u)$ is analytic in the open disc of radius $R$ around the origin where $R>1$, then $lim_{n \rightarrow \infty} a_n = f(1)$.
\end{lemma}

\begin{proof} Writing $f(u)=\sum_{n \geq 0} b_n u^n$, one has that $a_n = b_0 + \cdots + b_n$, and the result follows.
\end{proof}

Lemma \ref{pole} is a basic criterion for absolute convergence which replaces products with sums.

\begin{lemma} \label{pole} The following are equivalent for a domain $D$ in the complex plane:
\begin{enumerate}
\item $\prod (1+a_n)$ converges absolutely (and uniformly over $D$)
\item $\sum |a_n|$ converges absolutely (and uniformly over $D$) 
\end{enumerate}
\end{lemma}

\begin{lemma} \label{sec} (Identity 3.4 of \cite{BR1})  Let $\omega$ be a root of $z^{q-1}-1$ other than $1$. Then
\[ \prod_{\phi \in \Phi_q^+} \left( 1 - \omega^{r(\phi)} u^{deg(\phi)} \right) = 1.\]
\end{lemma}

In what follows we will need to use the cycle index $Z_{\SL_n(q)}$ of $\SL_n(q)$. For a set of matrices $S$, the cycle index $Z_S$ of $S$ is
the polynomial in variables
\[ \{ x_{\phi,\lambda}: \phi \in \Phi_q^+, \lambda \in \Lambda \} \] defined by
\[ Z_S = \frac{1}{|S|} \sum_{A \in S} \prod_{\phi \in \Phi_q^+} x_{\phi,\lambda_A(\phi)}, \] where $\lambda_A(\phi)$ is the partition corresponding
to the irreducible polynomial $\phi$ in the rational canonical form of $A$.

Define $C_{\phi,q}(\lambda)$ to be equal to the size of the centralizer of the block matrix
\[ diag(\gamma(\phi^{\lambda_1}),\gamma(\phi^{\lambda_2}),\cdots) \] in $\GL_{deg(\phi) |\lambda|}(q)$, or equal to $1$ if $\lambda=\emptyset$. Here $\gamma(\phi^{\lambda_i})$ is the companion matrix of $\phi^{\lambda_i}$.

From Stong \cite{Sto}, one has the following cycle index for the groups $\GL_n(q)$. Note that the partitions $\lambda$ range over all partitions of all non-negative integers, and that $|\lambda|$ denotes the size of the partition $\lambda$.

\[ 1 + \sum_{n \geq 1} Z_{\GL_n(q)} u^n = \prod_{\phi \in \Phi_q^+} \sum_{\lambda} \frac{x_{\phi,\lambda} u^{deg(\phi) |\lambda|}}
{C_{\phi,q}(\lambda)}.\] 

Britnell \cite{BR1} derived a cycle index for the groups $\SL_n(q)$. To state it, let $\Omega_{q-1}$ be the set of complex roots of $z^{q-1}-1$.
For $\omega \in \Omega_{q-1}$, define

\[ K_{\omega}(u) = \prod_{\phi \in \Phi_q^+} \sum_{\lambda} \frac{\omega^{r(\phi) |\lambda|} x_{\phi,\lambda} u^{deg(\phi) |\lambda|}}
{C_{\phi,q}(\lambda)}.\] 

Britnell's cycle index for the groups $\SL_n(q)$ is given by

\[ q-1 + \sum_{n \geq 1} Z_{\SL_n(q)} u^n = \sum_{\omega \in \Omega_{q-1}} K_{\omega}(u).\]

It what follows, we let $N(q;j)$ be the number of monic degree $j$ irreducible polynomials over $\F_q$ with non-zero constant term. It is known
that $N(q;1)=q-1$ and that for $j>1$, \[ N(q;j)=\frac{1}{j} \sum_{r|j} \mu(r) q^{j/r},\] where $\mu$ is the Moebius function of
elementary number theory.

\begin{theorem} \label{slcoset} Fix $t \geq 1$. For any coset of $\SL_n(q)$ in $\GL_n(q)$, the fixed $q$, $n \rightarrow \infty$ proportion of elements whose characteristic polynomial has no irreducible factors of degree at most $t$ converges to
\begin{equation} \label{produ} \prod_{j=1}^t \prod_{i \geq 1} \left( 1 - 1/q^{ij} \right)^{N(q;j)}.\end{equation} 
\end{theorem}

\begin{proof} We first prove the result for the coset $\SL_n(q)$. Since the argument is almost identical, we assume that $t=1$ and point out the amendments needed for general $t$.

In Britnell's cycle index for special linear groups, set $x_{\phi,\lambda}$ to equal 0 if the degree of $\phi$ is equal to 1 and $|\lambda|>0$, and to equal 1 otherwise. Then $Z_{\SL_n(q)}$ is the proportion of eigenvalue free elements of $\SL_n(q)$. With these same substitutions, we get that
\[ K_{\omega}(u) = \prod_{\phi: deg(\phi)>1} \sum_{\lambda} \frac{\omega^{r(\phi) |\lambda|} u^{deg(\phi) |\lambda|}}{C_{\phi,q}(\lambda)}.\]

Now  \[ \sum_{\lambda} \frac{\omega^{r(\phi) |\lambda|} u^{deg(\phi) |\lambda|}}{C_{\phi,q}(\lambda)} = \sum_{\lambda} \frac{\left( \omega^{r(\phi)} u^{deg(\phi)} \right)^{|\lambda|}}{C_{\phi,q}(\lambda)}.\]

From \cite{Sto}, one has that
\[ \sum_{\lambda} \frac{u^{deg(\phi) |\lambda|}}{C_{\phi,q}(\lambda)} = \prod_{i \geq 1} \frac{1}{1-u^{deg(\phi)}/q^{i deg(\phi)}}.\]
Thus
\[  \sum_{\lambda} \frac{\omega^{r(\phi) |\lambda|} u^{deg(\phi) |\lambda|}}{C_{\phi,q(\lambda)}} = 
 \prod_{i \geq 1} \frac{1}{1-u^{deg(\phi)} \omega^{r(\phi)}/q^{i deg(\phi)}}.\]

It follows that
\begin{eqnarray*}
K_{\omega}(u) & = & \prod_{\phi:deg(\phi)>1} \prod_{i \geq 1} \frac{1}{1-u^{deg(\phi)} \omega^{r(\phi)}/q^{i deg(\phi)}}\\
& = & \prod_{\phi:deg(\phi)=1} \prod_{i \geq 1} \left( 1 - u \omega^{r(\phi)}/q^i \right)
\prod_{\phi} \prod_{i \geq 1}  \frac{1}{1-u^{deg(\phi)} \omega^{r(\phi)}/q^{i deg(\phi)}}.
\end{eqnarray*}

Lemma \ref{sec} implies that for $\omega \neq 1$,
\[ \prod_{i \geq 1}   \prod_{\phi} \frac{1}{1-u^{deg(\phi)} \omega^{r(\phi)}/q^{i deg(\phi)}} = 1.\]

Hence if $\omega \neq 1$,
\[ K_{\omega}(u) = \prod_{i \geq 1} \prod_{\phi:deg(\phi)=1} \left( 1- u \omega^{r(\phi)}/q^i \right)
= \prod_{i \geq 1} \prod_{j=1}^{q-1} \left( 1 - u \omega^j/q^i  \right).\] 

Note that for general $t$, 

\[ K_{\omega}(u) = \prod_{i \geq 1} \prod_{j=1}^t \prod_{\phi:deg(\phi)=j} \left( 1- u^j \omega^{r(\phi)}/q^{ij} \right).\]

Now if $\omega=1$, we have that

\[ K_1(u) = \prod_{\phi:deg(\phi)=1} \prod_{i \geq 1} \left( 1 - u/q^i \right)
\prod_{\phi} \prod_{i \geq 1}  \frac{1}{1-u^{deg(\phi)}/q^{i deg(\phi)}}.\]
Setting all variables in the cycle index of the groups $\GL_n(q)$ equal to 1 gives that
\[ \prod_{\phi} \prod_{i \geq 1}  \frac{1}{1-u^{deg(\phi)}/q^{i deg(\phi)}} = \frac{1}{1-u}.\]
It follows that
\[ K_1(u) = \frac{1}{1-u} \left( \prod_{i \geq 1} (1-u/q^i) \right)^{q-1}.\]

Note that for general $t$,

\[ K_1(u) = \frac{1}{1-u} \left( \prod_{j=1}^t \prod_{i \geq 1} (1-u/q^{ij}) \right)^{N(q;j)}.\]

Summarizing, we have shown that with the above substitutions \[ q-1 + \sum_{n \geq 1} Z_{\SL_n(q)} u^n\] is equal to
$A+B$ where
\[ A =  \frac{1}{1-u} \prod_{i \geq 1} (1-u/q^i)^{q-1} \] and
\[ B = \frac{1-u}{1-u}  \sum_{\omega \neq 1} \prod_{i \geq 1} \prod_{j=1}^{q-1} (1- u \omega^j /q^i).\] (Note that there is no
misprint in the formula for $B$). The result for $\SL_n(q)$ now follows from Lemma \ref{fir}, since by Lemma \ref{pole},
$(1-u)A$ and $(1-u)B$ are both analytic in a disc of radius $R>1$ (any $1<R<q$ works).

To treat other cosets of $\SL_n(q)$ in $\GL_n(q)$, use the fact from \cite{BR1} that
\[ 1 + \sum_{n \geq 1} Z_{\mu \SL_n(q)} u^n = \sum_{\omega \in \Omega_{q-1}} \omega^{r(-\mu)} K_{\omega}(u) \]
and again the dominant term is $\omega=1$.
\end{proof}

Next we derive some useful properties of the quantity  \eqref{produ}. 

\begin{proposition} \label{boun} Consider \eqref{produ} with $t$ fixed.
\begin{enumerate}
\item The quantity \eqref{produ} converges to \[ e^{-(1+1/2+\cdots+1/t)} \] as $q \rightarrow \infty$. 
\item The quantity \eqref{produ} is uniformly bounded away from $1$.
\item  The quantity \eqref{produ} is uniformly bounded away from $0$.
\end{enumerate}
\end{proposition}

\begin{proof} Part 1 follows easily from the fact that the $i>1$ terms in \eqref{produ} converge to $1$.

Next we prove part 2. By looking only at the $i=j=1$ term, one has that
\begin{eqnarray*}
\prod_{j=1}^t \prod_{i \geq 1} \left( 1 - 1/q^{ij} \right)^{N(q;j)} & < & (1-1/q)^{q-1} \\
& = & e^{(q-1) \log(1-1/q)} \\
& \leq & e^{-(q-1)/q} \\
& \leq & \frac{1}{\sqrt{e}}.
\end{eqnarray*}

To prove part 3, since $t$ is fixed, it's enough to show that for $j$ fixed between $1$ and $t$,
\[ \prod_{i \geq 1} (1-1/q^{ij})^{N(q;j)} \] is uniformly bounded away from $0$. Since $N(q;j) \leq q^j/j$, one has
that \[ \prod_{i \geq 1} (1-1/q^{ij})^{N(q;j)} \geq \prod_{i \geq 1} (1-1/q^{ij})^{q^j/j} \geq \prod_{i \geq 1} (1-1/q^{ij})^{q^j}.\]
So it's enough to show that for all prime powers $q$, $\prod_i (1-1/q^i)^q$ is uniformly bounded away from $0$. Now
\begin{eqnarray*}
\prod_i (1-1/q^i)^q & = & e^{q \log \prod_i (1-1/q^i)} \\
& = & e^{q \sum_i \log (1-1/q^i)} \\
& \geq & e^{q \sum_i (-1/q^i - 1/q^{2i})} \\
& = & e^{-1/(1-1/q) - 1/(q(1-1/q^2))}. \end{eqnarray*} The inequality used the fact that $\log(1-x) \geq -x-x^2$ for $0 \leq x \leq 1/2$.
Finally, note that $-1/(1-1/q) - 1/(q(1-1/q^2))$ is minimized for $q=2$, so that \[ \prod_i (1-1/q^i)^q \geq e^{-2-2/3} > 0.\]
\end{proof} 

\subsection{Cosets of $\SU_n(q)$ in $\GU_n(q)$} \label{cosetsofu}

In fact the unitary analog of Section \ref{cosetA} is not needed for the proof of \Cref{t:main}. Nevertheless, we state the corresponding
result, whose proof runs along the same lines as Theorem \ref{slcoset}, using Britnell's cycle index for special unitary groups \cite{BR2} instead of his
cycle index for the special linear groups. We put the details in an appendix.

\begin{theorem} \label{ucoset}  For $t$ and $q$ fixed, the $n \rightarrow \infty$ limiting proportion
of elements in any coset of $\SU_n(q)$ in $\GU_n(q)$ which have no factors of degree at most $t$ dividing the characteristic
polynomial is equal to \begin{equation} \label{productu} \prod_{j=1}^t \prod_{i \geq 1} \left( 1 + \frac{1}{(-q)^{ij}} \right)^{\tilde{N}(q;j)}
\prod_{i \geq 1} \left( 1 - \frac{1}{q^{2ij}} \right)^{\tilde{M}(q;j)} \end{equation}
\end{theorem}
 Here $\tilde{N}(q;j)$ is the number of
monic irreducible polynomials of degree $j$ over $\F_{q^2}$ such that $\phi=\tilde{\phi}$; this is only non-zero for odd $j$. And $\tilde{M}(q;j)$ is the number of
(unordered) pairs $\{\phi,\tilde{\phi}\}$ of monic irreducible polynomials of degree $j$ over $\F_{q^2}$ such that $\phi \neq \tilde{\phi}$.
Here the conjugate $\tilde{\phi}$ of a degree $n$ polynomial with non-zero constant term is defined by
\[ \tilde{\phi}(z) := \phi(0)^{-\sigma} z^n \phi^{\sigma}(1/z),\] where $\phi^{\sigma}(z)$ replaces each coefficient of $\phi(z)$ by its
$q$th power.

\begin{remark}
	
Using explicit formulas for $\tilde{N}(q;j)$ and $\tilde{M}(q;j)$ in \cite{F1}, it is straightforward to see that \eqref{productu} converges
to \[ e^{-\left( \sum_{j=1 \atop odd}^t 1/j + 1/2 \sum_{j=1}^t 1/j \right)} \]  as $q \rightarrow \infty$ (note that the $i>1$ terms
in \eqref{productu} contribute nothing). Moreover, for $q$ a prime power, one can show, arguing as in Proposition \ref{boun}, that
\eqref{productu} is uniformly bounded away from $0$ and $1$. Since we don't need these two facts, we omit the details.

\end{remark}

\subsection{$\Sp_{2n}(q)$} \label{symplecticgroups}

We provide results for the symplectic groups. The proofs for the formulas for limiting proportions are a minor modification of the methods of Neumann and Praeger \cite{NP} to study eigenvalue free proportions, so we just state the formulas. We do include a complete proof that in odd characteristic these proportions are uniformly bounded away from $0$ and $1$ (the even characteristic case is almost identical which is not surprising since the formulas for these proportions are almost identical).

First consider the case of odd characteristic. In the statement of Theorem \ref{sympodd}, $N^*(q;j)$ is the number of monic,
irreducible self-conjugate polynomials of degree $j$ over $\F_q$, and $M^*(q;j)$ is the number of (unordered) conjugate pairs $\{\phi,\phi^*\}$
of monic, irreducible non-self conjugate polynomials of degree $j$ over $\F_q$, and the conjugate $\phi^*$ of a degree $n$ polynomial $\phi$ with non-zero
constant term is defined by
\[ \phi^*(z):= \phi(0)^{-1} z^n \phi(z^{-1}).\] Explicit  formulas for $N^*(q;j)$ and $M^*(q;j)$ are given in Lemma 1.3.16 in \cite{FNP}.

\begin{theorem} \label{sympodd}
\begin{enumerate}
\item For $t$ and odd $q$ fixed, the $n \rightarrow \infty$ limiting proportion of
elements of $\Sp_{2n}(q)$ whose characteristic polynomial has no factors of degree at most $t$ is equal to
\begin{equation} \label{oddsymprod} \prod_{i \geq 1} (1-1/q^{2i-1})^2 \prod_{j \leq t/2} \prod_{i \geq 1} \left( 1+(-1)^i/q^{ij} \right)^{N^*(q;2j)} 
\prod_{j \leq t} \prod_{i \geq 1} \left( 1-1/q^{ij} \right)^{M^*(q;j)}. \end{equation}

\item The quantity (\ref{oddsymprod}) is uniformly bounded away from $1$.

\item The quantity (\ref{oddsymprod}) is uniformly bounded away from $0$.
\end{enumerate}
\end{theorem}

\begin{proof} The proof of part 1 is a straightforward generalization of the method of \cite{NP} for the special case $t=1$.

For part 2, we first claim that \[ \prod_{j \leq t/2} \prod_{i \geq 1} \left( 1+(-1)^i/q^{ij} \right)^{N^*(q;2j)} \leq 1.\] Indeed, this follows from the fact that \[ \prod_{i \geq 1} \left( 1+(-1)^i/q^{ij} \right) \leq 1,\]  which in turn follows from the fact that for any prime power $q$, \[ (1-1/q^i)(1+1/q^{i+1}) \leq 1 \] for all $i$.

Thus it is enough to show that 
\[  \prod_{i \geq 1} (1-1/q^{2i-1})^2 \prod_{j \leq t} \prod_{i \geq 1} \left( 1-1/q^{ij} \right)^{M^*(q;j)} \] is uniformly bounded away from $1$. Taking only terms $j=i=1$ in the second product and using the fact that $M^*(q;1)=(q-3)/2$, it is enough to show that
\[ \prod_{i \geq 1} (1-1/q^{2i-1})^2 \cdot (1-1/q)^{(q-3)/2} \] is uniformly bounded away from $1$. If $q=3$ this is clear, so assume that $q \geq 5$. Then \[ (1-1/q)^{(q-3)/2} = e^{(q-3)/2 \log(1-1/q)} \leq e^{-(q-3)/(2q)} \leq e^{-.2},\] where we used the inequality $\log(1-x) \leq -x$. This completes the proof of part 2.

For part 3, define
\[ A = \prod_{i \geq 1} (1-1/q^{2i-1})^2 \]
\[ B = \prod_{j \leq t/2} \prod_{i \geq 1} \left( 1+(-1)^i/q^{ij} \right)^{N^*(q;2j)} \]
\[ C = \prod_{j \leq t} \prod_{i \geq 1} \left( 1-1/q^{ij} \right)^{M^*(q;j)} \]

It is enough to show that $A,B,C$ are uniformly bounded away from $0$.

To see that $A$ is uniformly bounded away from $0$, note that
\[ \prod_{i \geq 1} (1-1/q^{2i-1})^2 >  \prod_{i \geq 1} (1-1/q^i)^2 \geq \prod_i (1-1/q^i)^q \] which is uniformly
bounded away from $0$ by the proof of part 3 of Proposition \ref{boun}.

To treat $B$, note that since $t$ is fixed, it's enough to show that for all fixed $j$,
\[ \prod_i (1+(-1)^i/q^{ij})^{N^*(q;2j)} \] is uniformly bounded away from $0$. It's easy to see
that for all $i,j$, \[  \prod_i (1+(-1)^i/q^{ij}) > (1-1/q^j).\] So we need only show that for all fixed $j$,
\[ (1-1/q^j)^{N^*(q;2j)} \] is uniformly bounded away from $0$. Since $N^*(q;2j) \leq q^j/(2j)$, it follows that
\[ (1-1/q^j)^{N^*(q;2j)} \geq (1-1/q^j)^{q^j/(2j)} \geq (1-1/q^j)^{q^j/2},\] so we show that for $q \geq 3$, the quantity $(1-1/q)^{q/2}$ is
uniformly bounded away from $0$. However this is clear since
\[ (1-1/q)^{q/2} = e^{q/2 \log(1-1/q)} \geq e^{q/2(-1/q-1/q^2)} = e^{-1/2-1/(2q)} \geq e^{-2/3}.\] The inequality used the
fact that $\log(1-x) \geq -x-x^2$ for $0 \leq x \leq 1/2$.

To treat $C$, note that since $t$ is fixed, it's enough to show that for all fixed $j$,
\[ \prod_i (1-1/q^{ij})^{M^*(q;j)} \] is uniformly bounded away from $0$. Since $M^*(q;j) \leq q^{j}/(2j)$ it's enough to show that
$\prod_i (1-1/q^{ij})^{q^j/(2j)}$ is uniformly bounded away from $0$. But this follows from the fact, proved in part 3 of Proposition \ref{boun}, that for all prime powers $q$,  $\prod_i (1-1/q^{i})^{q/2}$ is bounded away from $0$.

\end{proof}

\begin{remark} Using the explicit formulas for $N^*(q;j)$ and $M^*(q;j)$ it is easy to see that the $q \rightarrow \infty$
limit of \eqref{oddsymprod} is equal to \[ e^{-1/2 \left( \sum_{j \leq t/2} 1/j + \sum_{j \leq t} 1/j \right)}.\] (Indeed, the $i>1$ terms contribute nothing as $q \rightarrow \infty$).
\end{remark}

Next we consider even characteristic. The results and proofs are almost identical to the odd characteristic case, so we omit the
proofs. Note that in Theorem \ref{sympeven}, the formulas for $N^*(q;j)$ and $M^*(q;j)$ in Lemma 1.3.16 in \cite{FNP} are slightly different than in odd characteristic.

\begin{theorem} \label{sympeven}
\begin{enumerate}
\item For $t$ and even $q$ fixed, the $n \rightarrow \infty$ limiting proportion of
elements of $\Sp_{2n}(q)$ whose characteristic polynomial has no factors of degree at most $t$ is equal to
\begin{equation} \label{evensymprod} \prod_{i \geq 1} (1-1/q^{2i-1}) \prod_{j \leq t/2} \prod_{i \geq 1} \left( 1+(-1)^i/q^{ij} \right)^{N^*(q;2j)} 
\prod_{j \leq t} \prod_{i \geq 1} \left( 1-1/q^{ij} \right)^{M^*(q;j)}. \end{equation} 

\item The quantity (\ref{evensymprod}) is uniformly bounded away from $1$.

\item The quantity (\ref{evensymprod}) is uniformly bounded away from $0$.
\end{enumerate}
\end{theorem}

\begin{remark} Using  the explicit formulas for $N^*(q;j)$ and $M^*(q;j)$  it is easy to see that the $q \rightarrow \infty$
limit of \eqref{evensymprod} is equal to \[ e^{-1/2 \left( \sum_{j \leq t/2} 1/j + \sum_{j \leq t} 1/j \right)}, \] the same as in the odd
characteristic case.
\end{remark}

\subsection{$\Or^{\pm}_{2n}(q)$}

\label{subsec:orthogonal_generating_functions}

Neumann and Praeger \cite{NP} calculate the fixed $q$, $n \rightarrow \infty$ limiting proportion of eigenvalue free elements of $\Or^{\pm}_{2n}(q)$.
They showed that the limiting proportion of eigenvalue free elements of $\Or^+_{2n}(q)$ is equal to the limiting proportion of eigenvalue free elements
of $\Or^-_{2n}(q)$, and that both of these proportions are equal to $1/2$ of the corresponding proportions in $\Sp_{2n}(q)$. This holds for both odd and
even characteristic.

One can fix $t$ and look more generally at the elements whose characteristic polynomial has no factors of degree at most $t$ (eigenvalue free is just
the special case $t=1$). Using their methods, one sees that the $n \rightarrow \infty$ limiting proportion of these elements is the same for $\Or^+_{2n}(q)$ and
for $\Or^-_{2n}(q)$, and that these proportions are both $1/2$ of the corresponding symplectic proportions, which we studied in Subsection \ref{symplecticgroups}.

\section{Group theory: preliminaries}
\label{sec:group_theory_preliminaries}

The following will play a key role in the proof of \Cref{t:main}.

\begin{lemma}
\label{l:expectation}
    Let $S$ be a nonabelian finite simple group, and let $x,y\in \Aut(S)$. Let $M$ be a maximal subgroup of $G=\gen{S,x,y}$ and let $A$ be a non-empty $S$-stable subset of $Sy$. Then, the probability that $x$ and a random element of $A$ are contained in a conjugate of $M$ is at most
    \[
    \fpr(x,G/M) \frac{1}{|A|} \sum_{g\in A}\fp(g, G/M).
    \]
\end{lemma}

\begin{proof}
	We may assume that the probability in the statement is nonzero, so $MS=G$ and $S$ acts transitively on the $G$-conjugates of $M$.    Assume $A$ is the disjoint union of the $S$-classes $C_1, \ldots, C_t$, with representatives $x_1, \ldots, x_t$. The probability that $x$ and a random element of $A$ are contained in a conjugate of $M$ is at most
    \begin{align*}
    \frac{1}{|A|}\fp(x,G/M)|A\cap M| &= \frac{1}{|A|}\fpr(x,G/M)|G:M|\sum_{i=1}^t|C_i\cap M| \\
    &=\fpr(x,G/M) \frac{1}{|A|} |G:M|\sum_{i=1}^t|C_i|\fpr(x_i, G/M) \\
    &= \fpr(x,G/M) \frac{1}{|A|} \sum_{g\in A}\fp(g, G/M).
    \end{align*}
This proves the lemma.
\end{proof}

Specializing $A=Sy$ we get the following lemma.

\begin{lemma}
\label{l:fixed_point_ratio}
    Let $S$ be a nonabelian finite simple group, let $x,y\in \Aut(S)$, and let $M$ be a maximal subgroup of $G=\gen{S,x,y}$. Then, the probability that $x$ and a random element of $Sy$ are contained in a conjugate of $M$ is at most $\fpr(x,G/M)$.
\end{lemma}

\begin{proof}
Recall that if $K$ is a finite transitive permutation group and $N$ is a transitive normal subgroup with $K=\gen{N,k}$, then the average number of fixed points in the coset $Nk$ is $1$. Assuming that $MS=G$, we then apply \Cref{l:expectation} to $N=S$, $k=y$ and $K=\gen{S,y}$, acting on the right cosets of $M$ in $G$.
\end{proof}

\subsection{Expected number of fixed points} As already remarked in the introduction, we are concerned with  bounds to the average
\begin{equation}
\label{eq:average_formula}
\frac{1}{|A|} \sum_{g\in A}\fp(g, G/M)
\end{equation}
appearing in \Cref{l:expectation}, in the case where  $S$ is a simple classical group, $M$ is a subspace stabilizer, and $A\subseteq Sy$ for various choices of $A$ and $y\in \Aut(S)$.

More precisely, we will let $A=A(t)$ be the set of elements of $Sy$ fixing no space of dimension at most $t$, and we will compute the asymptotic of \eqref{eq:average_formula} in the regime $q$ and $t$ fixed and $n\to \infty$. In doing so, we will use all results proved in \Cref{sec:generating_functions}, which we summarize in \Cref{l:resume_limit_proportion} below.

In this section all asymptotic symbols are referred to the limit $n\to \infty$. For example, a quantity $o_{q,t}(1)$ is a quantity tending to zero as $n\to\infty$ with $q$ and $t$ fixed.

\subsection{Expectation for symmetric groups}

As a way of motivation, let us begin with the case of symmetric groups. This is not used in the proof of \Cref{t:main}, but it is instructive to first see this case.

Let $1\le t < n/2$, let $A_n(t)$ be the set of elements of $S_n$ fixing no subset of size at most $t$, and let $a_n(t)=|A_n(t)|/n!$. Note that if $t$ is fixed then $a_\infty(t):=\lim_{n\to \infty} a_n(t)$ exists and it belongs to $(0,1)$. (This follows from the well known fact that, for $t$ fixed,  the vector of cycle lengths $(c_1, \ldots, c_t)$ of a random permutation is distributed as $n\to \infty$ as $(X_1, \ldots, X_t)$, the $X_i$ being independent Poisson random variables of parameter $1/i$.) 

\begin{lemma}
\label{l:symmetric_groups}
    Let $1\le t\le k<n/2$ and let $\Omega$ be the set of $k$-subsets of $\{1, \ldots, n\}$. Then
    \[
    \frac{1}{|A|}\sum_{g\in A}\fp(g,\Omega) = a_k(t) + o_t(1).
    \]
This quantity is bounded away from $1$ when $t$ is fixed and $n$ is large.
\end{lemma}

\begin{proof}

Let $X$ be a $k$-set and let $M\cong S_k\times S_{n-k}$ be the stabilizer of $X$. Then $A\cap M$ consists of the elements $(x,y)$ where $x\in S_k$ and $y\in S_{n-k}$ fix no set of size at most $t$. We get
    \begin{align*}
    \frac{|A\cap M|}{|A|} &= \frac{1}{|A|}a_k(t)k!\cdot a_{n-k}(t)(n-k)! \\
    &= \frac{a_k(t)}{|\Omega|}(1+o_{t}(1)),
    \end{align*}
where we used that $a_{n-k}(t) = a_n(t)(1+o_t(1))$.
Summing over all conjugates of $M$, we get the equality in the statement. For the last sentence in the statement, simply observe that $a_k(t)$ is bounded away from one as $n\to \infty$; this is clear if $k$ is bounded, and if $k$ is large it approaches $a_\infty(t)<1$.
\end{proof}

\subsection{Classical groups: notation and limiting proportions}
\label{subsec:average_classical}
Now we address the case of classical groups. It is convenient to work with the covers that act faithfully on the natural module; let $I$ be one of the symbols $\GL, \Sp, \GU, \Or^\varepsilon$, with $\varepsilon \in \{+,-,\circ\}$, so that $I_n(q)$ is a classical group. When $n$ is odd, we write $I_n(q) = \Or^\circ_n(q)$ or $\Or_n(q)$ according to notational convenience. In this section we simply assume $n\ge 2$ in all cases.

 Let a positive integer $t$ be fixed. We now introduce some convenient notation for the set of elements in a certain coset fixing no subspace of dimension at most $t$, taking care of some obvious restrictions in orthogonal groups.

 \begin{notation}
 \label{notation}
\begin{itemize}[$\diamond$]
	\item (Non-orthogonal groups.) We assume first $I\neq \Or^\varepsilon$, as this case requires somewhat separate notation.	Let $L_n(q)=\SL_n(q), \Sp_n(q), \SU_n(q)$ in the respective cases, and let us fix a coset $C$ of $L_n(q)$ in $I_n(q)$. Let $A_n(q,t,C,I)$ be the set of elements of $C$ fixing no subspace of dimension at most $t$, and set
	\[
	a_n(q,t,C,I):=\frac{|A_n(q,t,C,I)|}{|L_n(q)|}.
	\]
	We also denote by $A_n(q,t,I)$ the set of elements of $I_n(q)$ fixing no subspace of dimension at most $t$, and \[
a_n(q,t,I):= \frac{|A_n(q,t,I)|}{|I_n(q)|}.
\]
	When $\tau$ is the inverse-transpose automorphism of $\GL_n(q)$ and $C=\GL_n(q)\tau$, we let $A_n(q,t,C,\GL)$ be the set of elements $g\tau$ of $C$ such that $(g\tau)^2=gg^\tau$ fixes no subspace of dimension at most $t$ if $n$ is even, and fixes a vector and no other subspace of dimension at most $t$ if $n$ is odd. We set 	\[
	a_n(q,t,C,\GL):=\frac{|A_n(q,t,C,\GL)|}{|\GL_n(q)|}.
	\] 
	\item (Orthogonal groups.) Assume now $I=\Or^\varepsilon$ and set $L_n(q)=\SO^\varepsilon_n(q)$ if $q$ is odd and $L_n(q)=\Omega^\varepsilon_n(q)$ if $q$ is even.  When $n$ is even, we define $A_n(q,t,S,\Or^\varepsilon)$ as the set of elements of $\Or^\varepsilon_n(q)$ fixing no subspace of dimension at most $t$. We also define $A_n(q,t,O,\Or^\varepsilon)$ as the set of elements of $\Or^\varepsilon_n(q)$ acting as a reflection on a nondegenerate $2$-space, and belonging to $A_{n-2}(q,t,S,\Or^{\varepsilon'})$ on the perpendicular complement, where $\varepsilon'$ is determined. Observe that  $A_n(q,t,S,\Or^\varepsilon)\subseteq L_n(q)$ and $A_n(q,t,O,\Or^\varepsilon)\subseteq \Or^\varepsilon_n(q)\sm L_n(q)$. 
	
	When $n$ is odd, we define $A_n(q,t,S,\Or)$ (resp. $A_n(q,t,O,\Or)$) as the set of elements of $\Or_n(q)$ acting trivially (resp. as $-1$) on a nondegenerate $1$-space, and belonging to $A_{n-1}(q,t,S, \Or^{\varepsilon'})$ on the perpendicular complement, where $\varepsilon'$ is determined. Observe that $A_n(q,t,S,\Or)\subseteq L_n(q)$ and $A_n(q,t,O,\Or)\subseteq \Or_n(q)\sm L_n(q)$.
	
	In all cases, we set
 \[
 a_n(q,t,-,\Or^\varepsilon):=\frac{|A_n(q,t,-,\Or^\varepsilon)|}{|L_n(q)|}.
	\]

\end{itemize}
 \end{notation}

\begin{remark}
	\label{rem:omega}
	When $I=\Or^\varepsilon$ and $q$ is odd, the above definitions do not distinguish between the two cosets of $\Omega^\varepsilon_n(q)$ in $\SO^\varepsilon_n(q)$. We will take care of this issue later; see \Cref{subsec:omega}.
\end{remark}

In order to handle the coset $\GL_n(q)\tau$, where $\tau$ is the inverse-transpose automorphism, we will use a result from Fulman--Guralnick \cite{fulman2004guralnick_inverse_tranpose}. This result essentially says that statistics in the coset $\GL_n(q)\tau$ are equal to corresponding statistics in $\Sp_n(q)$ ($n$ even) or $\Sp_{n-1}(q)$ ($n$ odd), and so can be studied ``linearly''. In particular, with our notation they showed the following:

\begin{theorem}\cite[Corollaries 8.4 and 8.5]{fulman2004guralnick_inverse_tranpose}
	\label{t:inverse_transpose_fg}
	Let $C=\GL_n(q)\tau$. We have
	\[
	a_n(q,t,C,\GL) = a_{n-\delta}(q,t,\Sp),
	\]
	where $\delta=0$ if $n$ is even and $\delta=1$ if $n$ is odd.
\end{theorem}

In orthogonal groups, we also have that the proportion $a_n(q,t,O,\Or^\varepsilon)$ can be expressed in terms of $a_{n-\delta}(q,t,S,\Or^\varepsilon)$, where $\delta=(2,n)$.

\begin{lemma}
	\label{l:orthogonal_calculation}
	Assume $n\ge 5$ and $\varepsilon \in \{+,-,\circ\}$. Then
	\[
	a_n(q,t,O,\Or^\varepsilon)=\frac{a_{n-\delta}(q,t,S,\Or^+)+a_{n-\delta}(q,t,S,\Or^-)}{2}
	\]
	where $\delta=(2,n)$.
\end{lemma}

\begin{proof}
	Let us begin with the case $n$ even and $\varepsilon =+$. Set $L^\pm_n(q)=\SO^\pm_n(q)$ if $q$ is odd and  $L^\pm_n(q)=\Omega^\pm_n(q)$ if $q$ is even. Then 
\begin{align*}
	A_n(q,t,O,\Or^+)&=\sum_{W+} (q-1)a_{n-2}(q,t,S,\Or^+)|L^+_{n-2}(q)| + \sum_{W-} (q+1)a_{n-2}(q,t,S,\Or^-)|L^-_{n-2}(q)| \\ 
	&=|L^+_n(q)|\cdot \frac{a_{n-2}(q,t,S,\Or^+)+a_{n-2}(q,t,S,\Or^-)}{2},
\end{align*}
	where the first sum ranges over all nondegenerate $2$-spaces of plus type, and similarly the second. If $\varepsilon =-$ we get the same result by swapping signs suitably, and the case $n$ odd is very similar.
\end{proof}

We now record the results proved in \Cref{sec:generating_functions}, which will be crucial in the proof of \Cref{t:main}.

\begin{theorem}
\label{l:resume_limit_proportion}
    For $I\neq \Or^\varepsilon$, let $C$ be any coset of $L_n(q)$ in $I_n(q)$, or $C=\GL_n(q)\tau$ when $I=\GL$. For $I=\Or^\varepsilon$ let $C\in \{S,O\}$. Then
    \[
    a_\infty(q,t,C,I):=\lim_{n\to \infty} a_n(q,t,C,I)
    \]
    exists and belongs to $(0,1)$. Moreover, when $I\neq \Or^\varepsilon$ and $C$ is a coset of $L_n(q)$ in $I_n(q)$, for every $I$ the limit is independent of $C$, and finally $a_\infty(q,t,C,\Or^\varepsilon)=a_\infty(q,t,C',\Or^{\varepsilon'})$ for every $\varepsilon, \varepsilon'$ and for every $C,C'\in \{S,O\}$.
\end{theorem}

\begin{proof}
Let us begin with the case where $I\neq \Or^\varepsilon$ and $C$ is a coset of $L_n(q)$ in $I_n(q)$, or $I=\Or^\varepsilon$ with $n$ even and $C=S$. Noting the simple fact that $g\in \GL_n(q)$ fixes a $k$-space if and only if the characteristic polynomial of $g$ has a factor of degree $k$, the statement follows then from \Cref{slcoset,boun} (in case $I=\GL$), \Cref{ucoset} (in case $I=\GU$), \Cref{sympodd,sympeven} (in case $I=\Sp$ and $I=\Or^\varepsilon$, cf. \Cref{subsec:orthogonal_generating_functions}).

The case $I=\Or^\varepsilon$ and $C=O$ follows from the previous paragraph and \Cref{l:orthogonal_calculation}. Finally, the case  $C=\GL_n(q)\tau$ follows from the previous paragraph and \Cref{t:inverse_transpose_fg}. 
	\end{proof}

 \begin{remark}
     As already remarked in \Cref{cosetsofu}, in the proof of \Cref{t:main} for unitary groups we will not use \Cref{l:resume_limit_proportion}. Nonetheless, we include this case in \Cref{l:expectation_eigenvalue_free}, below, since the  result is interesting and the proof is the same as in the other cases.
 \end{remark}

\subsection{Expectation for classical groups}
 We now present analogues of \Cref{l:symmetric_groups} for classical groups. In the interest of clarity, we prefer to handle separately, in \Cref{l:expectation_orthogonal_groups} below, orthogonal groups.
 
For $I=\GL$ we view a $k$-space as a totally singular $k$-space. We exclude from the following lemma the case of totally singular $k$-spaces where $n/2-k$ is bounded; see \Cref{rem:totally_singular_large} for an explanation.

\begin{lemma}
\label{l:expectation_eigenvalue_free}
Let $I\neq \Or^\varepsilon$ and let $C$ be a coset of $L_n(q)$ in $I_n(q)$. Let $1\le t< k < n/2$, let $\Omega$ the set of totally singular or nondegenerate $k$-spaces, and let $W\in \Omega$. We assume that if $W$ is totally singular then $n/2-k\to \infty$ as $n\to \infty$. Then, setting $A:=A_n(q,t,C,I)$, we have
     \[
    \frac{1}{|A|}\sum_{g\in A} \fp(g,\Omega) = \ a_k(q^u,t,X) +o_{q,t}(1)
    \]
where 
\begin{itemize}[$\diamond$]
\item if $I=\GU$ and $W$ is totally singular then $u=2$, and otherwise $u=1$;
    \item if $W$ is totally singular then $X=\GL$, and if $W$ is nondegenerate then $X=I$.
\end{itemize}
In all cases, $a_k(q^u,t,X) +o_{q,t}(1)$ is bounded away from $1$ when $q$ and $t$ are fixed and $n$ is large.

\end{lemma}

\begin{proof}
Let us begin with the case where $W$ is totally singular. Let $M$ be the stabilizer in $I_n(q)$ of $W$, so $M=U\rtimes (\GL_k(q^u)\times I_{n-2k}(q))$, where $U\le L_n(q)$. Then $A\cap M$ consists of the elements $(u,x,y)$ with $u\in U$, $x\in \GL_k(q^u)$ fixes no space of dimension at most $t$, and $y\in I_{n-2k}(q)$ fixes no space of dimension at most $t$ and belongs to a coset of $L_{n-2k}(q)$ in $I_{n-2k}(q)$ determined by $x$. By \Cref{l:resume_limit_proportion} the proportion of such elements is independent of the coset up to $o_{q,t}(1)$. Moreover, by assumption $n-2k \to \infty$ and so we get 
    \begin{align*}
    \frac{|A\cap M|}{|A|} &= \frac{1}{|A|}a_k(q^u,t,\GL)|\GL_k(q^u)|\cdot a_{n-2k}(q,t,I)|L_{n-2k}(q)|(1+o_{q,t}(1))\cdot |U| \\
    &= \frac{a_k(q^u,t, \GL)}{|\Omega|}(1+o_{q,t}(1)).
    \end{align*}
Summing over all conjugates of $M$, we get the equality in the statement, and the last sentence holds as in \Cref{l:symmetric_groups}.

The case of nondegenerate subspaces is entirely analogous. In this case $n-k>n/2$ and so we do not need any assumption on $k$.
\end{proof}

\begin{remark}
	\label{rem:totally_singular_large}
	 We excluded from \Cref{l:expectation_eigenvalue_free} the case of totally singular $k$-spaces with $n/2-k$ bounded. In this case, the same proof gives that the expectation is bounded, and very likely an analysis of the limits $a_\infty(q,t,C,I)$ in \Cref{l:resume_limit_proportion} will show that the expectation is less than $1$ also in this case. We do not pursue this, as we do not need it for the proof of \Cref{t:main}. 
\end{remark}

Let us now move to orthogonal groups. Note that, when $n$ is even, no element of $A_n(q,t,C,\Or^\varepsilon)$ ($C\in \{S,O\}$) fixes a nondegenerate space of odd dimension, so we exclude this case from the statement. We also assume $t\ge 2$ in order to get a cleaner estimate -- this ensures that elements of $A_n(q,t,C,\Or^\varepsilon)$ fix at most $1$-space and at most one $2$-space.

\begin{lemma}
	\label{l:expectation_orthogonal_groups}
  Let $2\le t\le k \le n/2$, let $C\in \{S,O\}$,  let $\Omega$ be an $\Or^{\varepsilon}_n(q)$-orbit on totally singular or nondegenerate $k$-spaces, and let $W\in \Omega$.  We assume that if $W$ is totally singular then $n/2-k\to \infty$ as $n\to \infty$, and that if $W$ is nondegenerate and $n$ is even then $k$ is even. Then setting $A:=A_n(q,t,C,\Or^\varepsilon)$, we have
     \[
    \frac{1}{|A|}\sum_{g\in A} \fp(g,\Omega) = f +o_{q,t}(1)
    \]
where
\begin{itemize}[$\diamond$] 
    \item if $W$ is totally singular then $f= a_k(q,t,\GL)$;
    \item if $n$ is even, $C=S$ (resp. $C=O$) and $W$ is nondegenerate of type $\varepsilon'$ then  $f=\frac{1}{2}a_k(q,t,S,\Or^{\varepsilon'})$ (resp. $f=\frac{1}{2}a_k(q,t,S,\Or^{\varepsilon'}) + \frac{1}{2}a_k(q,t,O,\Or^{\varepsilon'})$);
    \item if $nk$ is odd and $W$ is nondegenerate then $f=\frac{1}{2}a_k(q,t,C,\Or)$;
    \item if $n$ is odd, $k$ is even and $W$ is nondegenerate of type $\varepsilon'$ then $f=\frac{1}{2}a_k(q,t,S,\Or^{\varepsilon'})$.
\end{itemize}
In all cases, $f +o_{q,t}(1)$ is bounded away from $1$ when $q$ and $t$ are fixed and $n$ is large. 
\end{lemma}

\begin{proof}
 Set $L^\delta_n(q) = \SO^\delta_n(q)$ if $q$ is odd and $L^\delta_n(q) = \Omega^\delta_n(q)$ if $q$ is even. The proof of the lemma is similar to the one of \Cref{l:expectation_eigenvalue_free}. The case where $n$ is even and $C=S$ is identical; we just note that we use that the limiting proportion is independent of the sign $\varepsilon$ and of $C$ (see the last part of \Cref{l:resume_limit_proportion}). 

 Let us move to the case where $n$ is even, $C=O$ and $W$ is totally singular. Let $M$ be the stabilizer in $\Or^\varepsilon_n(q)$ of $W$, so $M=U\rtimes (\GL_k(q)\times \Or^{\varepsilon}_{n-2k}(q))$, where $U\le \Omega^{\varepsilon}_n(q)$. Then $A\cap M$ consists of the elements  $(u,x,y)$ where $u\in U$, $x\in \GL_k(q)$ fixes no space of dimension at most $t$, and $y\in A_{n-2k}(q,t,O,\Or^{\varepsilon})$. (Observe that $x$ cannot fix any $1$-space or $2$-space on $W\oplus W^*$, since the elements of $B$ fix only one $1$-space and at most one $2$-space.) Using that $n-2k\to \infty$ and \Cref{l:resume_limit_proportion} we get
    \begin{align*}
    \frac{|A\cap M|}{|A|} &= \frac{1}{|A|}a_k(q,t,\GL)|\GL_k(q)|\cdot a_{n-2k}(q,t,O,\Or^{\varepsilon})|L^\varepsilon_{n-2k}(q)|\cdot |U| \\
    &= \frac{a_k(q,t, \GL)}{|\Omega|}(1+o_{q,t}(1)).
    \end{align*}
Summing over all conjugates of $M$, we get the result.

Assume now $n$ is even, $C=O$ and $W$ is nondegenerate, so $k$ is even and $M=\Or^{\varepsilon'}_k(q)\times \Or^{\delta}_{n-k}(q)$, where $\varepsilon'\delta= \varepsilon$. Then $A\cap M$ consists of the elements $(x,y)$, where $x\in A_k(q,t,X,\Or^{\varepsilon'})$ and  $y\in A_{n-k}(q,t,Y,\Or^{\delta})$ for distinct $X,Y\in \{S,O\}$. Therefore 
   \begin{align*}
    \frac{|A\cap M|}{|A|} &= \frac{1}{|A|}a_k(q,t,S,\Or^{\varepsilon'})|L^{\varepsilon'}_k(q)|\cdot a_{n-k}(q,t,O,\Or^{\delta})|L^{\delta}_{n-k}(q)| \\
    &+ \frac{1}{|A|}a_k(q,t,O,\Or^{\varepsilon'})|L^{\varepsilon'}_k(q)|\cdot a_{n-k}(q,t,S,\Or^{\delta})|L^{\delta}_{n-k}(q)| \\
    &= \frac{a_k(q,t,S, \Or^{\varepsilon'})+a_k(q,t,O,\Or^{\varepsilon'})}{2|\Omega|}(1+o_{q,t}(1)),
    \end{align*}
where we used \Cref{l:resume_limit_proportion} (and, as already remarked in the first paragraph, also the last part of the statement). Summing over all conjugates on $M$ we get the desired estimate.

In case $n$ odd, an analogous argument gives the desired estimate (in this case, we must choose an element with eigenvalue $1$ on the subspace of odd dimension).
\end{proof}

We now prove an analogue for the coset $C=\GL_n(q)\tau$, where $\tau$ denotes the inverse-transpose automorphism.

Recall that a \textit{flag} (resp. \textit{antiflag}) is a pair $\{U,W\}$ of nonzero subspaces of $\F_q^n$ with $\dim(U)\le \dim(W)$, $\dim(U)+\dim(W)=n$, and $U\le W$ (resp. $U\cap W=0$). We say that a flag or antiflag is a \textit{$k$-flag} or \textit{$k$-antiflag} if $\dim(U)=k$.

We observe that when $n$ is odd then every element of $A_n(q,t,C,\GL)$ fixes a $1$-antiflag. Moreover, when $n$ is even and $k$ is odd, then no element of $A_n(q,t,C,\GL)$ fixes a $k$-antiflag. We exclude these cases from the following lemma. 

\begin{lemma}
	\label{l:expectation_graph}
	Let $C=\GL_n(q)\tau$, where $\tau$ is the inverse-transpose automorphism. Let $1\le t< k < n/2$, let $\Omega$ be the set of $k$-flags or $k$-antiflags, and let $W\in \Omega$. If $W$ is an antiflag assume $k\ge 2$ if $n$ is odd, and $k$ even if $n$ is even; and if $W$ is a flag assume $n/2-k\to \infty$ as $n\to \infty$. Then, setting $A:=A_n(q,t,C,\GL)$, we have
	\[
	\frac{1}{|A|} \sum_{g\in A} \fp(g,\Omega) = a_{k-\delta}(q,t,I) + o_{q,t}(1)
	\]
	where
	\begin{itemize}[$\diamond$]
		\item if  $W$ is a flag then $I=\GL$ and $\delta=0$;
		\item if $W$ is an antiflag then $I=\Sp$, and moreover $\delta=0$ if $k$ is even and $\delta=1$ if $k$ is odd.
	\end{itemize}
\end{lemma}

\begin{proof}
	Let us begin with the case where $W$ is an antiflag. Let $M$ be the stabilizer of $W$ in $\gen{\GL_n(q),\tau}$; we have $M=\gen{H,\tau}$ where $H\cong \GL_k(q)\times \GL_{n-k}(q)$ and $\tau$ acts as the inverse-transpose automorphism on each factor. Note that $A\cap M$ consists of the elements $g\tau$, with $g=(g_1,g_2)\in H$ and $g_i$ is such that $g_ig_i^\tau=(g_i\tau)^2$ fixes no space of dimension at most $t$, or fixes a vector and no other space of dimension at most $t$ depending on the parity of $k$ and $n-k$.

	By \Cref{t:inverse_transpose_fg}, the number of choices for $g_1$ is $a_{k-\delta}(q,t,\Sp)|\GL_k(q)|$, and the number of choices for $g_2$ is $a_{n-k-\delta'}(q,t,\Sp)|\GL_{n-k}(q)|$, where $\delta'\in \{0,1\}$ is such that $n-k-\delta'$ is even. For the same reason $|A|=|a_{n-\delta''}(q,t,\Sp)|\GL_n(q)|$, where $\delta''\in \{0,1\}$ is such that $n-\delta''$ is even, and so by \Cref{l:resume_limit_proportion} we get
	\begin{align*}
		\frac{|A\cap M|}{|A|} &= \frac{1}{|A|} a_{k-\delta}(q,t,\Sp)|\GL_k(q)| \cdot a_{n-k-\delta'}(q,t,\Sp)|\GL_{n-k}(q)| \\
		&= \frac{a_{k-\delta}(q,t,\Sp)}{|\Omega|}(1+o_{q,t}(1)).
	\end{align*}
	Summing over all conjugates of $H$ gives the result.
	
	The case where $W$ is a flag is analogous; here a point stabilizer $M$ has the form $\gen{H,\tau}$ where $H$ is a parabolic of type $(k,n-2k,k)$ and $\tau$ acts on the diagonal blocks by $(g_1,g_2,g_3)^\tau = (g_3^\tau, g_2^\tau, g_1^\tau)$, so the diagonal blocks of $gg^\tau$ are $(g_1g_3^\tau, g_2g_2^\tau,g_3g_1^\tau)$. In particular if $g\tau\in A \cap M$ then $g_1$ is arbitrary, $(g_2\tau)^2$ fixes no space of dimension at most $t$, or a vector and no other space of dimension at most $t$ depending on the parity of $n$, and $g_3$ is such that $g_1g_3^\tau$ fixes no space of dimension at most $t$. We also use that $n-2k\to \infty$ as in the proof of \Cref{l:expectation_eigenvalue_free,l:expectation_orthogonal_groups}.
\end{proof}

\subsection{$\Omega_n(q)$ with $q$ odd}
\label{subsec:omega}
As already mentioned in \Cref{rem:omega}, in orthogonal groups in odd characteristic we have to distinguish between the cosets of $\Omega^\varepsilon_n(q)$ in $\Or^\varepsilon_n(q)$.

The key ingredient to handle this can be found in \cite{fulman2017guralnick_subspace}. An analogue, for random polynomials  rather than random elements, was proved in \cite[Proposition 2.8 and Lemma 3.11]{eberhard2023garzoni_boston_shalev} with other methods.

\begin{lemma}
	\label{l:weyl_omega_two_cosets}
	Let $q$ be odd. The proportion of elements of $\SO^\varepsilon_n(q)$ that are semisimple and whose characteristic polynomial has an even number of irreducible factors of degree $k$ for every $k\in 4\Zint$ tends to $0$ as $n\to \infty$.
\end{lemma}

\begin{proof}
	Let $W=W(C_{\floor{n/2}})$ be the Weyl group of type $C_{\floor{n/2}}$. It can be found in \cite[Lemma 7.22]{fulman2017guralnick_subspace} that the proportion of elements of $W$ having an even number of negative $k$-cycles for every even $k$ tends to zero as $n\to \infty$. Then the same is true for elements in the cosets $W(D_{\floor{n/2}})$ and $W\sm W(C_{\floor{n/2}})$. The statement, then, follows from known correspondences between maximal tori and Weyl group (see \cite[Section 5]{fulman2017guralnick_subspace} or \cite[Section 3.2]{garzoni2023mckemmie} for the details). 
\end{proof}

\begin{lemma}
	\label{l:omega_two_cosets}
	Let $q$ be odd. There exist subsets $X_\Omega \subseteq \Omega^\varepsilon_n(q)$ and $X_S \subseteq (\SO^\varepsilon_n(q) \sm \Omega^\varepsilon_n(q))$ and a bijection $\phi\colon X_\Omega \to X_S$ such that
	\begin{itemize}
		\item[(i)] $|X_\Omega|/|\Omega^\varepsilon_n(q)|\to 1$ as $n\to \infty$; and
		\item[(ii)] for each subspace $W$ of $\F_q^n$ and $g\in X_\Omega$, $W^g=W$ if and only if $W^{\phi(g)}=W$. 
	\end{itemize}
\end{lemma}

\begin{proof}
	For an element $g\in \SO^\varepsilon_n(q)$, let $V_g$ be the sum of all primary blocks on which $g$ is semisimple; this is a well-defined nondegenerate subspace of $\F_q^n$. By \cite[Theorem 8.1]{fulman2017guralnick_subspace}, with high probability as $n\to \infty$ we have $\dim(V_g) \ge n-\log(n)$ (say). We now let $X$ be the set of elements $g$ of $\SO^\varepsilon_n(q)$ such that $\dim(V_g)\ge n-\log(n)$ and such that the characteristic polynomial of $g|_{V_g}$ has an odd number of irreducible factors for some $k\in 4\Zint$. By \Cref{l:weyl_omega_two_cosets} we have $|X|/|\SO^\varepsilon_n(q)|\to 1$ as $n\to \infty$.

	We now set $X_\Omega=X\cap \Omega^\varepsilon_n(q)$ and $X_S = X \cap (\SO^\varepsilon_n(q) \sm \Omega^\varepsilon_n(q))$. The bijection $\phi$ is constructed as follows: For $g\in X_\Omega$, let $k\in 4\Zint$ be minimal such the characteristic polynomial of $g|_{V_g}$ has an odd number of irreducible factors. Let $W$ be the sum of all subspaces corresponding to such factors, and define $\phi(g)$ as the element obtained by multiplying $g$ by $-1$ on $W$. We have indeed $\phi(g)\in X_S$ (see \cite[Proposition 2.5.13]{kleidman1990liebeck}). Clearly $\phi$ is a bijection; (ii) is clear, and (i) follows since $|X_\Omega|/|\Omega^\varepsilon_n(q)| = |X|/|\SO^\varepsilon_n(q)|$.
\end{proof}

\subsection{A fixed point ratio estimate}

We need another result, which is a refinement of a fixed point ratio estimate from \cite{guralnick2000kantor}.

For later reference, let us begin by recording some of these estimates. If $G$ is an almost simple classical group with socle $S$ and natural module $V$ of dimension $n$, a subgroup $M$ of $G$ is called a \textit{subspace subgroup} if it is the stabilizer of a nonzero proper subspace of $V$ (of any relevant type), or $S=\PSL_n(q)$, $G\not\leq \PGamL_n(q)$ and $M$ is the stabilizer of a flag or antiflag, or $S=\Sp_n(q)$, $q$ even and $M=\Or^\pm_n(q)$.

In the following statement, for $1\le k \le n/2$ we say that a subspace subgroup is a \textit{$k$-subspace subgroup} if it is the stabilizer of a $k$-space or the stabilizer of a $k$-flag or $k$-antiflag. It is convenient to include the subgroups $\Or^\pm_n(q) < \Sp_n(q)$ among the $1$-subspace subgroups. We also set $u=2$ if $G$ is unitary and $u=1$ otherwise.

\begin{theorem}\cite{guralnick2000kantor}
	\label{t:guralnick_kantor}
	Let $G$ be an almost simple classical group with socle $S$ and with natural module $\F^n_{q^u}$ ($u\in \{1,2\}$), let $1\le k \le n/2$ and let $M$ be a $k$-subspace subgroup of $G$. Then, for every $1\neq g\in G$, the following hold:
	\begin{itemize}
		\item[(i)] if $S=\PSL_n(q)$ then $\fpr(g,G/M)< 2/q^k$, and if moreover $k=1$ then $\fpr(g,G/M)< 1/q+1/q^{n-1}$.
		\item[(ii)] If $S\not\cong \PSL_n(q)$ then $\fpr(g,G/M) < 1/q^{uk} + 1/q^{cn}$ for an absolute constant $c>0$.
	\end{itemize}
\end{theorem}

\begin{proof}
	(i) is contained in \cite[Proposition 3.1 and Lemma 3.12]{guralnick2000kantor}, and (ii) is contained in \cite[Propositions 3.15, 3.16 and 3.18]{guralnick2000kantor}. (In case of nondegenerate $k$-spaces, apply \cite[Proposition 3.16]{guralnick2000kantor} replacing $k$ by $n-k$.) 
\end{proof}

We need an improvement for $1$-flags and $1$-antiflags. Note that the fixed point ratio on $1$-spaces is bounded by approximately $1/q$, so the fixed point ratio on pairs ($1$-space, hyperplane) is bounded by approximately $1/q^2$. One does not expect any substantial perturbation in restricting to flags or antiflags; we now confirm this fact.

In the following lemma, we could decrease the constant $4$. We do not bother, as the bounds in \cite{guralnick2000kantor} are accurate also for small $|S|$. Moreover, we assume $n\ge 4$, just because for $n=3$ and $g$ a field automorphism of order $2$, the fixed point ratio on $1$-flags is roughly $q^{-3/2}$.

\begin{lemma}
\label{l:improvement_fpr}
Let $S=\PSL_n(q)$, with $n\ge 4$, and let $1\neq g\in \Aut(S)$. Let $F=F(n,q)$ be the set of $1$-flags, and let $A=A(n,q)$ be the set of $1$-antiflags. Then for $\Omega \in \{F,A\}$ we have
\[
\fpr(g,\Omega)< \frac{1}{q^2} + \frac{4}{q^{n-1}}.
\].
\end{lemma}

\begin{proof}
Let us record the size of the relevant sets:
\[
|F| = \frac{(q^n-1)(q^{n-1}-1)}{(q-1)^2} \hspace{15pt} |A| = \frac{(q^n-1)q^{n-1}}{q-1}
\]
We may work with a preimage of $g$ in $\GamL_n(q)\rtimes \gen \gamma$, where $\gamma$ is the inverse-transpose map, and we may assume $g$ has prime order modulo $\Z(\GL_n(q))$.

Let us assume first $g\in \GL_n(q)$ with $g$ semisimple. If $g$ fixes a $1$-space, then $g$ does not act homogeneously; so write $V=V_1\oplus V_2$ where there is no nontrivial $\gen g$-homomorphism from one space to the other. Say that $\dim(V_1)=n_1$ and $\dim(V_2)=n_2$ with $n_1\le n_2$. We may assume that $g$ acts scalarly on $V_1$ and $V_2$, in which case 
\begin{align*}
&\fp(g,F) = \frac{q^{n_1}-1}{q-1}\left(\frac{q^{n_1-1}-1}{q-1} + \frac{q^{n_2}-1}{q-1} \right) + \frac{q^{n_2}-1}{q-1} \left(\frac{q^{n_2-1}-1}{q-1} + \frac{q^{n_1}-1}{q-1}\right) \\
&\quad \le \frac{1}{(q-1)^2}\left((q^{n_2}-1)(q^{n_2-1}-1) 
+ 3(q^n-1)\right).
\end{align*}
(In the first equality, by convention $(q^{n_1-1}-1)/(q-1)=0$ if $n_1=1$.) Therefore
\[
\fpr(g,F) \le \frac{1}{q^{2(n-n_2)}} + \frac{3}{q^{n-1}-1} \le \frac{1}{q^2} + \frac{4}{q^{n-1}}.
\]
One proves similarly the bound for $\fpr(g,A)$.

Assume now $g\in \GL_n(q)$ with $g$ unipotent. Set $m=\dim([V,g])$. We count fixed points on $A$. Consider a $1$-space $W$ fixed by $g$. If there exists a complement $W'$ of $W$ fixed by $g$, then $[V,g]=[W',g]\le W'$, so $W\cap [V,g]=0$ and the number of choices for $W'$ is equal to the number of complements of $(W+[V,g])/[V,g]$ in $V/[V,g]$, namely $q^{n-m-1}$. The number of choices for $W$ is at most $(q^{n-m}-1)/(q-1)$,
so
\[
\fp(g,A) \le \frac{(q^{n-m}-1)q^{n-m-1}}{q-1}
\]
from which
\[
\fpr(g,A) \le \frac{q^{n-m}-1}{(q^n-1)q^m} < \frac{1}{q^{2m}}\le \frac{1}{q^2}.
\]
The bound for $\fpr(g,F)$ is similar.

If $g$ is a field automorphism, say of order $r$, then the fixed points on $F$ (resp. $A$) can be identified with $F(n,q_0)$ (resp. $A(n,q_0)$) where $q=q_0^r$ (see for example \cite[proof of Lemma 3.6]{guralnick2000kantor}), and one easily gets the desired bound.

If $g$ is a graph-field automorphism of order $2$, then fixed points on $F$ (resp. $A$) can be identified with singular (resp. nondegenerate) $1$-spaces of a unitary geometry over $\F_q$. If $g$ is a graph automorphism of order $2$, then fixed points on $F$ (resp. $A$) can be identified with singular (resp. nondegenerate) $1$-spaces for a bilinear form over $\F_q$. In particular, in both cases the number of fixed points is at most $(q^n-1)/(q-1)$, so $\fpr(g,F)$ and $\fpr(g,A)$ are at most $1/q^{n-2}$. 
\end{proof}

\section{Proof of \Cref{t:main}}

We are finally ready to prove \Cref{t:main}. The case of groups of bounded rank, as well as the case of large rank and non-subspace subgroups, follow directly from results of Liebeck--Saxl \cite{liebeck1991saxl}, Liebeck--Shalev \cite{liebeck1999shalev}, Guralnick--Larsen--Tiep \cite{guralnick2012larsen_tiep}.

\begin{theorem}
	\label{t:bounded_rank}
	\Cref{t:main} holds if $S$ is a group of Lie type of rank $r$ over $\F_q$ and $q$ is sufficiently large compared to $r$.
\end{theorem}

\begin{proof}
	Let $M$ be a maximal subgroup of $G=\gen{S,x,y}$ containing $x$ and some element of $Sy$, so $SM=G$. By \cite{liebeck1991saxl}, $\fpr(x,G/M)\ll 1/q^{1/2}$ (and in fact, if $S\not\cong \PSL_2(q)$ then $\fpr(x,G/M)\ll 1/q$). Moreover, from \cite[Theorem 1.2]{guralnick2012larsen_tiep} the number of conjugacy classes of maximal subgroups of $G$ is $\ll r^6 + r\log\log(q)$. Therefore, by \Cref{l:fixed_point_ratio}, summing over all conjugacy classes of maximal subgroups of $G$ containing $x$ and an element of $Sy$, we get that the probability that $x$ and a random element of $Sy$ do not generate $G$ is $\ll_r \log\log(q)/q^{1/2}$, which tends to zero as $q\to \infty$ with $r$ fixed.
\end{proof}

We will now work with classical groups of large rank. Let $\F_{q^u}^n$ be the natural module, where $u=2$ if $S$ is unitary and $u=1$ otherwise.

We now choose $A\subseteq Sy$ suitably. For the reader's convenience, we list here our choices. The reader is advised to move directly to the proof of \Cref{t:subspace}, where the motivation for the choices will be apparent and the argument will be easier to follow -- this list should only be used as a reference. See \Cref{notation} for the notation used.

See also the discussion around \eqref{eq:intro_kantor}, which highlights the idea of the proof of \Cref{t:subspace}.

\begin{choice}
\begin{itemize}[$\diamond$]
	\item If $S=\SL_n(2)$, we choose $A=A_n(2,3,\GL)$. (Note this holds for both the cases $y\in S$ and $y\not\in S$.)
	\item If $S=\PSL_n(3)$ and $y\in \PGL_n(3)$, we choose $A$ as the image modulo scalars of $A_n(3,1,C, \GL)$, where $C$ is a coset lifting $Sy$.
	\item If $S=\PSL_n(3)$ and $y\not\in\PGL_n(3)$, we choose $A=Sy$.
	\item If $S=\Sp_n(2)$ we choose $A=A_n(2,2,\Sp)$.
	\item If $S=\Omega^\pm_n(2)$ and $y\in S$, we choose $A=A_n(2,2,S,\Or^\pm)$.
	\item If $S=\Omega^\pm_n(2)$ and $y\not\in S$, we choose $A=A_n(2,4,O,\Or^\pm)$.
	\item If $S=\POm^\varepsilon_n(3)$, set $L=\PSO^\varepsilon_n(3)$ and $O=\PO^\varepsilon_n(3)$. 	If $n$ is even and $y\in L$, we choose $A$ as the image of $A_n(3,1,S,\Or^\varepsilon)\cap X_C$, where $C\in \{\Omega,S\}$ is determined by the condition $A\subseteq Sy$ (see \Cref{l:omega_two_cosets} for the notation $X_\Omega$ and $X_S$).
	\item If $S=\POm^\varepsilon_n(3)$ with $n$ even and 	$y\in O\sm L$ then we choose $A$ as the image of the set of elements acting as a reflection on a nondegenerate $2$-space, and belonging to $A_{n-2}(3,3,S,\Or^\pm)\cap X_C$ on the perpendicular complement, where $C\in \{\Omega,S\}$ is determined by the condition $A\subseteq Sy$.
	\item If $S=\POm^\varepsilon_n(3)$ with $n$ odd and $y\in L$ (resp. $y\in O\sm L$) then we choose $A$ as the image of the set of elements acting trivially (resp. as $-1$) on a nondegenerate $1$-space, and belonging to $A_{n-1}(3,3,S,\Or^\pm)\cap X_C$ on the perpendicular complement, where $C\in \{\Omega,S\}$ is determined by the condition $A\subseteq Sy$.
	\item In all other cases we choose $A=Sy$.
\end{itemize}

\end{choice}

Since by \Cref{l:resume_limit_proportion} $|A|/|S|\gg 1$, it is enough to prove \Cref{t:main} by picking random elements from $A$ (note also that $A=Sy$ if $q\ge 4$). In particular, we want to show that the probability $P$ that $x$ and a random element of $A$ do not generate $\gen{S,x,y}$ is bounded away from $1$.  
We have 
\begin{equation}
\label{eq:final_bound}
P\le P_1 + P_2,
\end{equation}
where $P_1$ is the probability that $x$ and a random element of $A$ are contained in a non-subspace core-free maximal subgroup of $\gen{S,x,y}$, or a $k$-subspace subgroup with $n/3< k\le n/2$; and $P_2$ is the probability that $x$ and a random element of $A$ are contained in a $k$-subspace subgroup with $1\le k\le n/3$. (The reason why we include subspace subgroups of large dimension into $P_1$ is merely formal, and relies on the fact that in \Cref{l:expectation_eigenvalue_free,l:expectation_orthogonal_groups,l:expectation_graph}  we have excluded the case of totally singular spaces of dimension $n/2-O(1)$; see \Cref{rem:totally_singular_large} for comments.) 

\begin{theorem}
	\label{t:non_subspace}
	$P_1\le q^{-c n}$ where $c>0$ is absolute.	
\end{theorem}

\begin{proof}
If $M$ is non-subspace subgroup, then by \cite[Theorem 1.1]{liebeck1999shalev} we have $\fpr(x,G/M) \le q^{-c' n}$ for some absolute $c' >0$. The same is true for a $k$-subspace subgroup with $n/3<k\le n/2$ by \Cref{t:guralnick_kantor}. Since by \cite{guralnick2012larsen_tiep} the number of conjugacy classes of maximal subgroups is $\ll n^6 + n\log \log q$, by \Cref{l:fixed_point_ratio} and a union bound we get that the probability that $x$ and a random element of $Sy$ are contained in such a subgroup is $\le q^{-c n}$ where $c$ is absolute. Since $|A|/|S|\gg 1$ by \Cref{l:resume_limit_proportion}, we get the same bound for $P_1$.
\end{proof}

Let us now move to the estimation of $P_2$. We use all the results from \Cref{sec:generating_functions,sec:group_theory_preliminaries}. 

\begin{theorem}
\label{t:subspace}
If $n$ is sufficiently large then $P_2<0.91$.
\end{theorem} 
\begin{proof}
We consider in turn all the families of simple classical groups. 

\textbf{Case 1:} $S=\PSL_n(q)$. If $q\ge 4$  we have $A=Sy$, and by \Cref{t:guralnick_kantor}, \Cref{l:fixed_point_ratio} and a union bound over all conjugacy classes of $k$-subspace subgroups with $1\le k \le n/3$ we get
\begin{equation}
	\label{eq:SL_n}
	P_2\le 2\left(\frac{1}{q} + \frac{1}{q^{n-1}} + \sum_{k=2}^{\infty} \frac{2}{q^k} \right) = \frac{2}{q} +  \frac{4}{q(q-1)} + \frac{2}{q^{n-1}}.
\end{equation}
(The factor $2$ at the beginning comes for the following reason: If $\gen{S,x,y}\le \PGamL_n(q)$, the number of fixed $k$-spaces equals the number of fixed $(n-k)$-spaces; if $\gen{S,x,y}\not\le \PGamL_n(q)$, we are counting both flags and antiflags.). Since $q\ge 4$, the right-hand side of \eqref{eq:SL_n} is $\le 5/6 + O(q^{-cn})$, and we are done. 

Assume now $S=\SL_n(2)$. For clarity in the exposition, we will keep using the variable $q$, and only in the end specialize to $q=2$.  Possibly replacing $y$ by $xy$, we assume that if $y\in S$ then $x\in S$. (We do this simply because we did not state a lemma bounding the expectation on flags or antiflags for the case $A\subseteq S$.) By \Cref{l:expectation_eigenvalue_free,l:expectation_graph}, if $n$ is large then the expected number of fixed points of an element of $A$ is at most $1$. Moreover, by the definition of $A$, clearly it is zero when $1\le k \le 3$, except for $1$-antiflags, where it is equal to $1$ when $n$ is odd and equal to $0$ when $n$ is even.

In particular, we can subtract $2/q+4/q^2 + 4/q^3$ from \eqref{eq:SL_n}, and by \Cref{l:improvement_fpr} add $1/q^2 + 4q^{1-n}$ coming from $1$-antiflags. Put it otherwise, by a union bound,  \Cref{l:expectation,l:expectation_eigenvalue_free,l:improvement_fpr} we get
\[
P_2\le \frac{1}{q^2}+ \frac{4}{q^{n-1}} + 2\sum_{k=4}^{\infty} \frac{2}{q^k} = \frac{1}{q^2} + \frac{4}{q^3(q-1)} + \frac{4}{q^{n-1}}  = \frac{3}{4} + \frac{4}{2^{n-1}} 
\]
where in the last equality we specialized $q=2$.

Assume now $S=\PSL_n(3)$. This case is similar. If $y\in \PGL_n(3)$ by the definition of $A$ we can substract $2/q$ from \eqref{eq:SL_n}, and we get a bound of $2/3$. If $y\not\in \PGL_n(3)$ we have $A=Sy$; by \Cref{l:improvement_fpr}, we may replace $2/q$ by $2/q^2$ in \eqref{eq:SL_n} and get a bound $8/9 + O(3^{-n})$.

We now proceed with the other cases. In the estimates below, the constant $c$ appearing denotes a positive absolute constant.

\textbf{Case 2:} $S=\PSU_n(q)$.  Here $A=Sy$ and by \Cref{l:fixed_point_ratio,t:guralnick_kantor} and a union bound we get 
\[
P_2\le 2\sum_{k=1}^{\infty} \frac{1}{q^{2k}} + O(n/q^{cn}) =  \frac{2}{q^2-1} + O(n/q^{cn}) \le \frac{2}{3}+O(nq^{-cn}).
\]
(The factor $2$ accounts for the fact that we are considering both totally singular and nondegenerate spaces.) 

\textbf{Case 3:} $S=\PSp_n(q)$. If $q\ge 3$ then $A=Sy$. Since nondegenerate spaces have even dimension, by \Cref{l:fixed_point_ratio,t:guralnick_kantor} and a union bound we get
\begin{align}
	\label{eq:Sp}
	P_2 &\le 2\delta/q+\sum_{k=1}^{\infty} \frac{1}{q^k} +\sum_{k=1}^{\infty} \frac{1}{q^{2k}}  + O(nq^{-cn}) \\
	&= 2\delta/q + \frac{1}{q-1}+\frac{1}{q^2-1} + O(nq^{-cn}) \nonumber 
\end{align}
where $\delta=1$ if $q$ is even, and $\delta=0$ if $q$ is odd. (The term $2\delta/q$ accounts for subgroups $\Or^\pm_n(q)<\Sp_n(q)$ with $q$ even.) Since $q\ge 3$, the quantity is $\le 9/10 + O(q^{-cn})$.

Assume now $S=\Sp_n(2)$. Note that elements of $A$ fix exactly one nondegenerate hyperplane, so 
\[
\frac{1}{|A|}\sum_{g\in A} \left(\fp(g,\Omega_+) + \fp(g,\Omega_-)\right) = 1,
\]
where $\Omega_\pm$ denotes the set of nondegenerate hyperplane of $\pm$ type. Moreover, elements of $A$ fix no singular $1$-space, no singular $2$-space and no nondegenerate $2$-space, so in view of \Cref{l:expectation,l:expectation_eigenvalue_free} we can subtract from \eqref{eq:Sp} $1/q + 1/q + 1/q^2 + 1/q^2$, and we get $P_2 \le 5/6 + O(n2^{-cn})$.

\textbf{Case 4:} $S=\POm^\varepsilon_n(q)$ with $\varepsilon\in\{+,-,\circ\}$. Assume first $q\ge 4$, so $A=Sy$. By \Cref{l:fixed_point_ratio,t:guralnick_kantor} we get 
\begin{align}
	\label{eq:fpr_orthogonal}
	P_2&\le \frac{\delta}{q}+ \sum_{k=1}^{\infty}\frac{1}{q^k}   +2\sum_{k=1}^{\infty} \frac{1}{q^{2k}} + \gamma\sum_{k=0}^{\infty} \frac{1}{q^{2k+1}} + O(nq^{-cn})  \\
	&= \frac{\delta}{q} + \frac{1}{q-1}+\frac{2+\gamma q}{q^2-1}  + O(nq^{-cn}),\nonumber
\end{align}
where $\delta=1$ if $q$ even and $\delta=0$ if $q$ is odd; and $\gamma=0$ if $q$ is even and $\gamma=2$ if $q$ is odd. (Let us make some comments on these parameters. The term $\delta/q$ accounts for nonsingular $1$-spaces with $q$ even. The parameter $\gamma$ is equal to the number of $S$-classes of maximal subgroups that are stabilizers of nondegenerate $k$-spaces with $k$ odd. Note that for $q$ odd and $n$ even these are fused in $\Aut(S)$, while for $q$ odd and $n$ odd they are not.) Since $q\ge 4$ this quantity is $\le 3/4 + O(q^{-cn})$, and we are done.

Assume now $S=\Omega^\pm_n(2)$. If $y\in S$, elements of $A$ fix no nonsingular $1$-space, singular $1$-space, singular $2$-space, nondegenerate $2$-space, so by \Cref{l:expectation,l:expectation_orthogonal_groups} we can subtract $1/q + 1/q+1/q^2+2/q^2$ from \eqref{eq:fpr_orthogonal} and we get $P_2 \le 5/12 +  O(nq^{-cn})$.  If $y\not\in S$, then elements of $A$ fix exactly one singular $1$-space, exactly one nondegenerate $2$-space, and no other space of dimension at most $4$. In particular we can subtract $1/q + 1/q^2 + 1/q^3 + 1/q^4 +1/q^2 + 2/q^4$, and we get $P_2 \le 41/48 +  O(nq^{-cn})$.

Assume finally $S=\POm^\varepsilon_n(3)$. Set $L=\PSO^\varepsilon_n(3)$ and $O=\PO^\varepsilon_n(3)$. Let $X_\Omega$ and $X_S$ be the sets from \Cref{l:omega_two_cosets}.

Assume first $n$ is even and $y\in L$. Set $A_\Omega=X_\Omega \cap A_n(3,1,S,\Or^\varepsilon)$ and  $A_S=X_S \cap A_n(3,1,S,\Or^\varepsilon)$ (note that $A=A_\Omega$ or $A_S$ depending on $y$). Setting $X=A_\Omega\cup A_S$ and   $Y=A_n(q,1,S,\Or^\varepsilon)$, \Cref{l:omega_two_cosets} implies $|A_\Omega|=|A_S|$, and moreover for each $C\in \{\Omega,S\}$  we have
\begin{align}
	\label{eq:expectation_omega}
	&\frac{1}{|A_C|} \sum_{g\in A_C} \fp(g,G/M)\nonumber
	= \frac{1}{|X|} \sum_{g\in X} \fp(g,G/M) \\
	&\quad\le \frac{1|Y|}{|X||Y|} \sum_{g\in Y} \fp(g,G/M).
\end{align}
(Here $M$ is any subspace subgroup of $G=\gen{S,x,y}$.) By \Cref{l:resume_limit_proportion,l:omega_two_cosets} we have 
$|Y|/|X|\to 1$. Therefore, in view of \Cref{l:expectation_eigenvalue_free}, \eqref{eq:expectation_omega} is bounded away from $1$. This is all we need; noting that an element of $A_C$ does not fix any $1$-space and any nondegenerate space of odd dimension, by \Cref{l:expectation,l:expectation_orthogonal_groups}  we can subtract from \eqref{eq:fpr_orthogonal} $1/q+2\sum_k 1/q^{2k+1}$, thereby getting $P_2 \le 5/12 + O(q^{-cn})$.

The other cases are essentially identical.  If $n$ is even and $y\in O\sm L$ the elements of $A$ fix exactly one singular $1$-space, exactly one nondegenerate $2$-space, no other space of dimension at most $3$, and no nondegenerate space of odd dimension, and subtracting the relevant quantities from \eqref{eq:fpr_orthogonal} we get $P_2 < 1/2 + O(q^{-cn})$.

If $n$ is odd, the elements of $A$ fix one nondegenerate $1$-space and no other space of dimension at most $3$, and we get $P_2 < 1/2 + O(q^{-cn})$.

The last case to consider is $n$ even and $y\in \Aut(S)\sm O$. In this case, $A=Sy$, and no element of $Sy$ fixes a nondegenerate space of odd dimension, so we calculate $P_2 \le 3/4 + O(nq^{-cn})$. This concludes the proof in all cases.
\end{proof}

\Cref{t:main} follows now immediately from \Cref{t:bounded_rank,t:non_subspace,t:subspace} and \eqref{eq:final_bound}. 

\section{Appendix: proof of Theorem \ref{ucoset}}
\label{appendix}

In this appendix we sketch the proof of Theorem \ref{ucoset}. 

We require some definitions (in addition to those in Section \ref{cosetsofu}). We let $\Phi_{q^2}^+$ denote the set of monic irreducible polynomials in $\F_{q^2}[z]$ other than the polynomial $z$. We let $\zeta$ be a generator of the cyclic subgroup of order $q+1$ in $\F_{q^2}^{\times}$. For a polynomial $\phi \in \Phi_{q^2}^+$, define
$s(\phi) \in Z_{q+1}$ by 
\begin{equation*}  \zeta^{s(\phi)} = \left\{ \begin{array}{ll}
		(-1)^{deg(\phi)} \phi(0)& \mbox{if $\phi=\tilde{\phi}$} \\
		\phi(0) \tilde{\phi}(0) & \mbox{if $\phi \neq \tilde{\phi}$}
	\end{array} \right. \end{equation*}

As earlier in the paper, we let $\Lambda$ denote the set of partitions of all non-negative integers $n$. The following lemma will be useful.

\begin{lemma} \label{equalone} (Identity 4.2 in \cite{BR2}) Let $\omega$ be a root of $z^{q+1}-1$ other than $1$. Then
	\[ \prod_{\phi \in \Phi_{q^2}^+ \atop \phi=\tilde{\phi}} (1-\omega^{s(\phi)} u^{deg(\phi)})
	\prod_{\{\phi,\tilde{\phi}\} \atop \phi \in \Phi_{q^2}^+, \phi \neq \tilde{\phi}} (1-\omega^{s(\phi)} u^{2deg(\phi)}) = 1.\]
\end{lemma}

In what follows we will need to use the cycle index $Z_{\SU_n(q)}$ of $\SU_n(q)$. For a set of matrices $S$ with entries in $\F_{q^2}$, the cycle index $Z_S$ of $S$ is
the polynomial in variables
\[ \{ x_{\phi,\lambda}: \phi \in \Phi_{q^2}^+, \lambda \in \Lambda \} \] defined by
\[ Z_S = \frac{1}{|S|} \sum_{A \in S} \prod_{\phi \in \Phi_{q^2}^+} x_{\phi,\lambda_A(\phi)}, \] where $\lambda_A(\phi)$ is the partition corresponding
to the irreducible polynomial $\phi$ in the rational canonical form of $A$.

For $\omega \in \Omega_{q+1}$ (the complex roots of $z^{q+1}-1$) we define $K_{U,\omega}(u)$ by

\[ R_{\omega}(u) := \prod_{\phi \in \Phi_{q^2}^+ \atop \phi=\tilde{\phi}} \left( 1 + \sum_{\lambda \in \Lambda} \frac{\omega^{s(\phi)|\lambda|} x_{\phi,\lambda} u^{deg(\phi)|\lambda|}}{C_{U,\phi}(\lambda)} \right) \]

\[ S_{\omega}(u) := \prod_{\{\phi,\tilde{\phi}\} \atop \phi \in \Phi_{q^2}^+,\phi \neq \tilde{\phi}} \left(  1 + \sum_{\lambda \in \Lambda}  \frac{\omega^{s(\phi) |\lambda|} x_{\phi,\lambda} x_{\tilde{\phi},\lambda} u^{2 deg(\phi) |\lambda|}}{C_{U,\phi}(\lambda)} \right) \]

\[ K_{U,\omega} := R_{\omega(u)} S_{\omega(u)}.\] Here $C_{U,\phi}(\lambda)$ are centralizer sizes of certain elements of unitary groups (the exact definition need not concern us), first calculated by Wall \cite{W}.  

When $\omega=1$, $K_{u,1}$ is Fulman's cycle index for the unitary groups.

Britnell's cycle index for the groups $\SU_n(q)$ \cite{BR2} is given by
\[ q+1 + \sum_{n \geq 1} Z_{\SU_n(q)} u^n = \sum_{\omega \in \Omega_{q+1}} K_{U,\omega}(u).\]

We now sketch the proof of Theorem \ref{ucoset}.

\begin{proof} Since we don't need the theorem for the proof of \Cref{t:main}, we write out full details for the case $t=1$ and for the coset $\SU_n(q)$.
	The other cases are quite similar, and note that if the limiting proportion exists for all cosets of $\SU_n(q)$, then it must equal the limiting proportion for
	$\SU_n(q)$  (taking the subsequence with $gcd(n,q+1)=1$ all cosets have the same number of elements we seek; just multiply by a scalar of norm $1$).
	
	In Britnell's cycle index for the special unitary groups, set $x_{\phi,\lambda}$ to equal $0$ if the degree of $\phi$ is equal to $1$ and $|\lambda|>0$, and to equal $1$ otherwise. Then $Z_{\SU_n(q)}$ is the proportion of eigenvalue free elements of $\SU_n(q)$. 
	
	Arguing as in the case of $\SL_n(q)$, one obtains that for all $\omega \in \Omega_{q+1}$,
	\[ K_{U,\omega} = ABCD, \] where
	\[ A:= \prod_{\phi=\tilde{\phi} \atop deg(\phi)=1} \prod_{i \ odd} \left( 1 - \frac{u \omega^{s(\phi)}}{q^i} \right)
	\prod_{i \ even} \left( 1 + \frac{u \omega^{s(\phi)}}{q^i}  \right)  \]
	
	\[ B:= \prod_{\{\phi,\tilde{\phi}\}, \phi \neq \tilde{\phi} \atop deg(\phi)=1} \prod_{i \ odd} \left( 1 - \frac{u^2 \omega^{s(\phi)}}{q^{2i}} \right)
	\prod_{i \ even} \left( 1 - \frac{u^2 \omega^{s(\phi)}}{q^{2i}}  \right)  \]
	
	\[ C:= \prod_{\phi = \tilde{\phi}} \prod_{i \ odd} \left( \frac{1}{1-u^{deg(\phi)} \omega^{s(\phi)}/q^{i deg(\phi)}} \right)
	\prod_{i \ even} \left( \frac{1}{1+u^{deg(\phi)} \omega^{s(\phi)}/q^{i deg(\phi)}} \right) \]
	
	\[ D:= \prod_{\{\phi,\tilde{\phi}\} \atop \phi \neq \tilde{\phi}} \prod_{i \ odd} \left( \frac{1}{1-u^{2deg(\phi)} \omega^{s(\phi)}/q^{2i deg(\phi)}} \right)
	\prod_{i \ even} \left( \frac{1}{1-u^{2deg(\phi)} \omega^{s(\phi)}/q^{2ideg(\phi)}} \right).\]
	
	Note that the products in $A,B,C,D$ are over irreducible polynomials.
	
	Setting $u=u/q^i$ in Lemma \ref{equalone} gives that for $\omega \neq 1$,
	\[ \prod_{i \ odd} \prod_{\phi=\tilde{\phi}} \frac{1}{1-u^{deg(\phi)} \omega^{s(\phi)}/q^{i deg(\phi)}}
	\prod_{\{\phi,\tilde{\phi}\} \atop \phi \neq \tilde{\phi}} \frac{1}{1-u^{2deg(\phi)} \omega^{s(\phi)}/q^{2i deg(\phi)}} = 1.\]
	
	Setting $u=-u/q^i$ in Lemma \ref{equalone} and using the fact that $\phi=\tilde{\phi}$ implies that $deg(\phi)$ is odd,
	gives that for $\omega \neq 1$,
	\[ \prod_{i \ even}  \prod_{\phi=\tilde{\phi}} \frac{1}{1+u^{deg(\phi)} \omega^{s(\phi)}/q^{i deg(\phi)}}
	\prod_{\{\phi,\tilde{\phi}\} \atop \phi \neq \tilde{\phi}} \frac{1}{1-u^{2deg(\phi)} \omega^{s(\phi)}/q^{2i deg(\phi)}} = 1.\]
	
	We conclude that if $\omega \neq 1$, then $CD=1$, so $K_{U,\omega}=AB$. 
	
	If $\omega=1$, then setting all variables in the cycle index of the groups $\GU_n(q)$ gives that $CD= \frac{1}{1-u}$. 
	Now the number of monic irreducible degree 1 polynomials $\phi$ with $\phi=\tilde{\phi}$ is equal to $q+1$, and the number
	of unordered  pairs $\{\phi,\tilde{\phi}\}$ such that $\phi \neq \tilde{\phi}$ is equal to $(q^2-q-2)/2$. Thus
	\[ K_{U,1} = \frac{1}{1-u} \prod_i \left( 1 + \frac{u}{(-q)^i} \right)^{q+1} \prod_i \left( 1-\frac{u^2}{q^{2i}} \right)^{(q^2-q-2)/2}.\]
	
	Summarizing, we have shown that with the above substitutions, \[ q+1 + \sum_{n \geq 1} Z_{\SU_n(q)} u^n \] is equal to
	the sum of 
	\[  \frac{1}{1-u} \prod_i \left( 1 + \frac{u}{(-q)^i} \right)^{q+1} \prod_i \left( 1-\frac{u^2}{q^{2i}} \right)^{(q^2-q-2)/2}, \] and
	
	\[ \sum_{\omega \neq 1} \prod_{\phi=\tilde{\phi} \atop deg(\phi)=1} \prod_i \left(1+u \omega^{s(\phi)}/(-q)^i \right)
	\prod_{\{\phi,\tilde{\phi}\}, \phi \neq \tilde{\phi} \atop deg(\phi)=1} \prod_i \left( 1-u^2 \omega^{s(\phi)}/q^{2i} \right).\]  The result for the coset $\SU_n(q)$ is now straightforward from Lemmas \ref{fir} and \ref{pole}.
\end{proof}


\begin{thebibliography}{AAA}

\bibitem [BGK]{breuer2008probabilistic}  Breuer, T.,  Guralnick, R. M., and Kantor, W. M., Probabilistic generation of finite simple groups, II, {\it J. Algebra}, {\bf 320} (2008), 443-494.

\bibitem [Br1]{BR1} Britnell, J., Cyclic, separable and semisimple matrices in the special linear groups over a finite field,
{\it J. Lond. Math. Soc.} {\bf 66} (2002), 605-622.

\bibitem [Br2]{BR2} Britnell, J., Cyclic, separable and semisimple transformations in the special unitary groups over a finite
field, {\it J. Group Theory} {\bf 9} (2006), 547-569.






\bibitem[BGH]{BGH}  Burness, T. C., Guralnick, R. M., and Harper, S., The spread of a finite group, {\it Ann. of Math.} {\bf 193} (2021), 619-687.

\bibitem[BGH2]{BGH2}  Burness, T. C., Guralnick, R. M., and Harper, S., Probabilistic $\frac{3}{2}$-generation of finite groups, preprint. 

\bibitem [DVL]{dallavolta1995lucchini} Dalla Volta, F., and Lucchini, A., Generation of almost simple groups, {\it J. Algebra} {\bf 2} (1995), 194-223.

\bibitem [Di]{Di}   Dixon, J. D., The probability of generating the symmetric group, {\it Math. Z.} {\bf 110} (1969), 199-205.


\bibitem [EG]{eberhard2023garzoni_boston_shalev} Eberhard, S., and Garzoni, D., Conjugacy classes of derangements in finite groups of Lie type, {\it Trans. Amer. Math. Soc., to appear}  (2024).

\bibitem [Fu]{F1} Fulman, J., Cycle indices for the finite classical groups, {\it J. Group Theory} {\bf 2} (1999), 251-289.


\bibitem [FG1]{fulman2004guralnick_inverse_tranpose} Fulman, J., and and Guralnick, R. M., Conjugacy class properties of the extension of $GL(n, q)$ generated by the inverse transpose involution, {\it J. Algebra} {\bf 275} (2004), 356-396.



\bibitem [FG2]{fulman2017guralnick_subspace} Fulman, J., and Guralnick, R. M., Derangements in subspace actions of finite classical groups, {\it Trans. Amer. Math. Soc.} {\bf 369} (1999), 2521-2572.



\bibitem [FNP]{FNP} Fulman, J., Neumann, P. M., and Praeger, C. E., A generating function approach to the enumeration of matrices
in classical groups over finite fields, {\it Mem. Amer. Math. Soc.}, Number 830, (2005).

\bibitem [GM]{garzoni2023mckemmie} Garzoni, D., and McKemmie, E., On the probability of generating invariably a finite simple group, {\it J. Pure Appl. Algebra} {\bf 227} (2023). 


\bibitem [GK]{guralnick2000kantor} Guralnick, R. M., and Kantor, W. M., Probabilistic generation of finite simple groups, {\it J. Algebra} {\bf 234} (2000), 743-792.


\bibitem [GKS]{guralnick1994probability} Guralnick, R. M., Kantor, W. M., and Saxl, J., The probability of generating a classical group, {\it Comm. Algebra} {\bf 22} (1994), 1395-1402.

\bibitem [GLT]{guralnick2012larsen_tiep} Guralnick, R. M., Larsen, M., and Tiep, P. H., Representation growth in positive characteristic and conjugacy classes of maximal subgroups, {\it Duke Math. J.} {\bf 161} (2012), 107-137.


\bibitem [KLu]{KL} Kantor, W. M., and Lubotzky, A., The probability of generating a finite classical group, {\it Geom. Dedicata} {\bf 36} (1990), 67-87.

\bibitem [KLi]{kleidman1990liebeck} Kleidman, P. B., and Liebeck, M. W., The subgroup structure of the finite classical groups, {\it Cambridge University Press}, {\bf 129} (1990).

 

\bibitem [LSa]{liebeck1991saxl} Liebeck, M. W., and Saxl, J., Minimal degrees of primitive permutation groups, with an application to monodromy groups of covers of Riemann surfaces, {\it Proc. Lond. Math. Soc.} {\bf 3} (1991), 266-314.


\bibitem[LS1]{LSh}  Liebeck, M. W., and Shalev, A., The probability of generating a finite simple group, {\it Geom. Dedicata}  {\bf 56} (1995), 103--113.


\bibitem[LS2]{liebeck1999shalev}  Liebeck, M. W., and Shalev, A., Simple groups, permutation groups, and probability, {\it J. Amer. Math. Soc.}  {\bf 12} (1999), 497--520.

\bibitem[Lu]{Lu}  Lucchini, A., Solubilizers in profinite groups, {\it J. Algebra} {\bf 647} (2024), 619-632.

\bibitem [NP]{NP} Neumann, P. M., and Praeger, C. E., Derangements and eigenvalue-free elements in finite classical groups,
{\it J. Lond. Math. Soc.} {\bf 58} (1998), 564-586.

\bibitem[St]{St}  Steinberg, R., Generators for simple groups, {\it Canadian J. Math.} {\bf 14} (1962), 277-283.

\bibitem [Sto]{Sto} Stong, R., Some asymptotic results on finite vector spaces, {\it Adv. Appl. Math.} {\bf 9} (1988), 167-199.

\bibitem [Wa]{W} Wall, G. E., On the conjugacy classes in the unitary, symplectic, and orthogonal groups, {\it J. Aust. Math. Soc.}
{\bf 3} (1963), 1-63.
\end{thebibliography}
\end{document}